\documentclass[a4paper,12pt,reqno]{amsart}
\usepackage[margin=1.3in,marginparwidth=1.5cm, marginparsep=0.5cm]{geometry}

\pdfoutput=1

\usepackage{booktabs} 

\usepackage{microtype}

\usepackage{mathrsfs}

\usepackage{MnSymbol}
\usepackage{color}

\usepackage{marginnote}
\usepackage{scalerel} 
\usepackage{tikz}

\colorlet{symbols}{black}
\colorlet{testcolor}{green!60!black}

\def\1{\mathbf{{1}}}

\usetikzlibrary{shapes.misc}
\usetikzlibrary{shapes.symbols}
\usetikzlibrary{decorations}
\usetikzlibrary{decorations.markings}

\def\drawx{\draw[-,solid] (-3pt,-3pt) -- (3pt,3pt);\draw[-,solid] (-3pt,3pt) -- (3pt,-3pt);}
\tikzset{
	root/.style={circle,fill=testcolor,inner sep=0pt, minimum size=2mm},
	dot/.style={circle,fill=black,inner sep=0pt, minimum size=1mm},
	var/.style={circle,fill=black!10,draw=black,inner sep=0pt, minimum size=2mm},
	circ/.style={circle,fill=white,draw=black,inner sep=0pt, minimum size=1.2mm},
	dotred/.style={circle,fill=black!50,inner sep=0pt, minimum size=2mm},
	generic/.style={semithick,shorten >=1pt,shorten <=1pt},
	gepsilon/.style={semithick,shorten >=1pt,shorten <=1pt,densely dashed},
	dist/.style={ultra thick,draw=testcolor,shorten >=1pt,shorten <=1pt},
	testfcn/.style={ultra thick,testcolor,shorten >=1pt,shorten <=1pt,<-},
	testfcnx/.style={ultra thick,testcolor,shorten >=1pt,shorten <=1pt,<-,
		postaction={decorate,decoration={markings,mark=at position 0.6 with {\drawx}}}},
	kprime/.style={semithick,shorten >=1pt,shorten <=1pt,dotted,->},
	kprimex/.style={semithick,shorten >=1pt,shorten <=1pt,densely dashed,->,
		postaction={decorate,decoration={markings,mark=at position 0.4 with {\drawx}}}},
	kernel/.style={semithick,shorten >=1pt,shorten <=1pt,->},
	multx/.style={shorten >=1pt,shorten <=1pt,
		postaction={decorate,decoration={markings,mark=at position 0.5 with {\drawx}}}},
	kernelx/.style={semithick,shorten >=1pt,shorten <=1pt,->,
		postaction={decorate,decoration={markings,mark=at position 0.4 with {\drawx}}}},
	kepsilon/.style={semithick,shorten >=1pt,shorten <=1pt,densely dashed,->},
	kernel1/.style={->,semithick,shorten >=1pt,shorten <=1pt,postaction={decorate,decoration={markings,mark=at position 0.45 with {\draw[-] (0,-0.1) -- (0,0.1);}}}},
	kernel2/.style={->,semithick,shorten >=1pt,shorten <=1pt,postaction={decorate,decoration={markings,mark=at position 0.45 with {\draw[-] (0.05,-0.1) -- (0.05,0.1);\draw[-] (-0.05,-0.1) -- (-0.05,0.1);}}}},
	kernelBig/.style={semithick,shorten >=1pt,shorten <=1pt,decorate, decoration={zigzag,amplitude=1.5pt,segment length = 3pt,pre length=2pt,post length=2pt}},
	rho/.style={dotted,semithick,shorten >=1pt,shorten <=1pt},
	renorm/.style={shape=circle,fill=white,inner sep=1pt},
	labl/.style={shape=rectangle,fill=white,inner sep=1pt},
	xi/.style={circle,fill=symbols!10,draw=symbols,inner sep=0pt,minimum size=1.2mm},
	xix/.style={crosscircle,fill=symbols!10,draw=symbols,inner sep=0pt,minimum size=1.2mm},
	xib/.style={circle,fill=symbols!10,draw=symbols,inner sep=0pt,minimum size=1.6mm},
	xibx/.style={crosscircle,fill=symbols!10,draw=symbols,inner sep=0pt,minimum size=1.6mm},
	not/.style={circle,fill=symbols,draw=symbols,inner sep=0pt,minimum size=0.5mm},
	>=stealth,
	}

\makeatletter
\def\DeclareSymbol#1#2#3{\expandafter\gdef\csname MH@symb@#1\endcsname{\tikz[baseline=#2,scale=0.15,draw=symbols]{#3}}\expandafter\gdef\csname MH@symb@#1s\endcsname{\scalebox{0.7}{\tikz[baseline=#2,scale=0.15,draw=symbols]{#3}}}}
\def\<#1>{\csname MH@symb@#1\endcsname}
\makeatother

\newcommand{\pe}{\mathbin{\scaleobj{0.7}{\tikz \draw (0,0) node[shape=circle,draw,inner sep=0pt,minimum size=8.5pt] {\footnotesize $=$};}}}

\newcommand{\pl}{\mathbin{\scaleobj{0.7}{\tikz \draw (0,0) node[shape=circle,draw,inner sep=0pt,minimum size=8.5pt] {\footnotesize $<$};}}}
\newcommand{\pg}{\mathbin{\scaleobj{0.7}{\tikz \draw (0,0) node[shape=circle,draw,inner sep=0pt,minimum size=8.5pt] {\footnotesize $>$};}}}

\DeclareSymbol{1}{0}{\draw[white] (-.4,0) -- (.4,0); \draw (0,0)  -- (0,1.2) node[dot] {};}
\DeclareSymbol{2}{0}{\draw (-0.5,1.2) node[dot] {} -- (0,0) -- (0.5,1.2) node[dot] {};}
\DeclareSymbol{3}{0}{\draw (0,0) -- (0,1.2) node[dot] {}; \draw (-.7,1) node[dot] {} -- (0,0) -- (.7,1) node[dot] {};}
\DeclareSymbol{30}{-3}{\draw (0,0) -- (0,-1); \draw (0,0) -- (0,1.2) node[dot] {}; \draw (-.7,1) node[dot] {} -- (0,0) -- (.7,1) node[dot] {};}
\DeclareSymbol{31}{-3}{\draw (0,0) -- (0,-1) -- (1,0) node[dot] {}; \draw (0,0) -- (0,1.2) node[dot] {}; \draw (-.7,1) node[dot] {} -- (0,0) -- (.7,1) node[dot] {};}
\DeclareSymbol{32}{-3}{\draw (0,0) -- (0,-1) -- (1,0) node[dot] {}; \draw (0,0) -- (0,-1) -- (-1,0) node[dot] {}; \draw (0,0) -- (0,1.2) node[dot] {}; \draw (-.7,1) node[dot] {} -- (0,0) -- (.7,1) node[dot] {};}
\DeclareSymbol{20}{-3}{\draw (0,0) -- (0,-1);\draw (-.7,1) node[dot] {} -- (0,0) -- (.7,1) node[dot] {};}
\DeclareSymbol{22}{-3}{\draw (0,0.3) -- (0,-1) -- (1,0) node[dot] {}; \draw (0,0.3) -- (0,-1) -- (-1,0) node[dot] {};\draw (-.7,1) node[dot] {} -- (0,0.3) -- (.7,1) node[dot] {};}
\DeclareSymbol{31p}{-3}{\draw (0,0) -- (0,-1) -- (1,0) node[dot] {}; \draw (0,0) -- (0,1.2) node[dot] {}; \draw (-.7,1) node[dot] {} -- (0,0) -- (.7,1) node[dot] {}; \draw (0,-1) node{\scaleobj{0.5}{\pe}}; }
\DeclareSymbol{32p}{-3}{\draw (0,0) -- (0,-1) -- (1,0) node[dot] {}; \draw (0,0) -- (0,-1) -- (-1,0) node[dot] {}; \draw (0,0) -- (0,1.2) node[dot] {}; \draw (-.7,1) node[dot] {} -- (0,0) -- (.7,1) node[dot] {}; \draw (0,-1) node{\scaleobj{0.5}{\pe}};}
\DeclareSymbol{22p}{-3}{\draw (0,0.3) -- (0,-1) -- (1,0) node[dot] {}; \draw (0,0.3) -- (0,-1) -- (-1,0) node[dot] {};\draw (-.7,1) node[dot] {} -- (0,0.3) -- (.7,1) node[dot] {}; \draw (0,-1) node{\scaleobj{0.5}{\pe}};}

\usepackage[colorlinks=true, pdfstartview=FitV, linkcolor=blue, citecolor=blue, urlcolor=blue,pagebackref=false]{hyperref}

\let\emptyset \undefined
\newsymbol\emptyset    203F

\theoremstyle{plain}
\newtheorem{theorem}{Theorem}[section]

\newtheorem{lemma}[theorem]{Lemma}
\newtheorem{proposition}[theorem]{Proposition}

\newtheorem{definition}[theorem]{Definition}

\theoremstyle{remark}
\newtheorem{remark}[theorem]{Remark}

\numberwithin{equation}{section}
\numberwithin{table}{section}

\newcommand{\D}{\mathbf{D}}
\newcommand{\E}{\mathbb{E}}
\newcommand{\N}{\mathbb{N}}
\renewcommand{\P}{{\mathbb P}}

\newcommand{\R}{\mathbb{R}}
\newcommand{\T}{\mathbb{T}}
\newcommand{\Z}{\mathbb{Z}}

\newcommand{\Cc}{\mathcal{C}}
\newcommand{\Ss}{\mathcal{S}}

\newcommand{\al}{\alpha}
\newcommand{\be}{\beta}

\newcommand{\de}{\delta	}
\newcommand{\eps}{\varepsilon}

\newcommand{\om}{\omega}
\newcommand{\ph}{\varphi}

\def\scal#1{\langle#1\rangle}

\DeclareMathOperator{\supp}{Supp}
\newcommand{\ls}{\lesssim}
\renewcommand{\subset}{\subseteq}

\renewcommand{\d}{\mathrm{d}}

\newcommand{\td}{\widetilde}
\newcommand{\mcl}{\mathcal}

\newcommand{\Ll}{\left}
\newcommand{\Rr}{\right}
\newcommand{\cc}{\mathbf{c}}
\newcommand{\F}{\mathscr{F}}
\newcommand{\dk}{\delta_k}
\renewcommand{\hat}{\widehat}
\newcommand{\htau}{\widehat{\tau}}

\newcommand{\hP}{\widehat{P}}
 								
\newcommand{\Ca}{\Cc^{|\tau|}}

\begin{document}

\title{Construction of $\Phi^4_3$ diagrams for pedestrians}

\author{Jean-Christophe Mourrat \and Hendrik Weber \and Weijun Xu}

\address[Jean-Christophe Mourrat]{Ecole normale sup\'erieure de Lyon, CNRS, Lyon, France}
\email{jean-christophe.mourrat@ens-lyon.fr}

\address[Hendrik Weber]{University of Warwick, Coventry, United Kingdom}
\email{hendrik.weber@warwick.ac.uk}

\address[Weijun Xu]{University of Warwick, Coventry, United Kingdom}
\email{weijun.xu@warwick.ac.uk}



\begin{abstract}
We aim to give a pedagogic and essentially self-contained presentation of the construction of various stochastic objects appearing in the dynamical $\Phi^4_3$ model. The construction presented here is based on the use of paraproducts. The emphasis is on describing the stochastic objects themselves rather than introducing a solution theory for the equation. 
\end{abstract}

\maketitle

\section{Introduction}
The purpose of this note is to give a pedagogic presentation of the construction of the various stochastic ``basis'' terms entering the construction of the dynamic $\Phi^4$ theory in three space dimensions ($\Phi^4_3$ for short). Formally, the dynamic $\Phi^4$ model on the torus $\T^3 = [0,1]^3$ is the solution $X(t,x)$ to the stochastic partial differential equation
\begin{equation}
\label{e:eqX}
\Ll\{
\begin{array}{ll}
\partial_t X = \Delta X - X^{3} + m X +  \xi, \qquad \text{on } \R_+ \times [0,1]^3, \\
X(0,\cdot) = X_0,
\end{array}
\Rr.
\end{equation}
where  $\xi$ denotes a white noise over $\R \times \T^3$, and $m$ is a real parameter.  Equation 
\eqref{e:eqX} describes the natural reversible dynamics for the ``static'' $\Phi^4_3$ Euclidean field theory, which is formally given by the expression 
\begin{equation}\label{QFT-measure}
\mu \propto \exp\Ll(- 2 \int_{\T^3} \Ll[\frac 1 2 | \nabla X|^2 + \frac{1}{4}X^4 -  \frac m 2 X^2 \Rr] \Rr) \prod_{x \in \T^3} \mathrm d X(x).
\end{equation}
The quartic potential $X^4$ in this energy gives the model its name (replacing $X$ by $\Phi$). Mathematically, neither \eqref{e:eqX} nor \eqref{QFT-measure} make sense as they stand. While this problem is the main concern of this note, we postpone its discussion and first proceed heuristically. (Alternatively, we temporarily replace the continuous space $\T^3$ by a finite approximation, with a suitable interpretation of the gradient.) 

The potential function $x \mapsto \frac 1 4 x^4 - \frac m 2 x^2$ has a single minimum at $x = 0$ for $m \le 0$. As $m$ becomes positive, a pitchfork bifurcation occurs, with the appearance of two minima at $x = \pm \sqrt{m}$, while the point $x = 0$ becomes a local maximum. In the energy between square brackets appearing in \eqref{QFT-measure}, the part consisting of
\begin{equation*}  
\int_{\T^3} \Ll[\frac{1}{4}X^4 -  \frac m 2 X^2 \Rr]
\end{equation*}
favors fields $X$ that take values close to those minima, while the part
\begin{equation*}  
\int_{\T^3} |\nabla X|^2
\end{equation*}
favors some agreement between nearby values of the field $X$. This description is highly reminiscent of that of the Ising model. Indeed, these two models are believed to have comparable phase transitions and large-scale properties. In one and two space dimensions (when $\T^{3}$ is replaced by $\T^{d}$, $d \in \{1,2\}$), a precise link between the Glauber dynamics of an Ising model with long-range interactions and the dynamic $\Phi^4$ model was obtained \cite{BPRS,FR,Ising}, and a similar result is conjectured to hold in our present three-dimensional setting.

Starting from the 60's, the $\Phi^4$ model was the subject of an intense research effort from the perspective of quantum field theory. From this point of view, the construction of the so-called Euclidean $\Phi^{4}$ measure \eqref{QFT-measure} is a first step towards building the corresponding quantum field theory. This requires the verification of certain properties known as the Osterwalder-Schrader axioms \cite{os1,os2}, among which the reflection positivity and the invariance under Euclidean isometry are the most important (we refer to \cite{reflection} for a review on reflection positivity --- in particular, the Ising measure is reflection positive, see \cite[Corollary~5.4]{reflection}). The whole endeavour was viewed as a test-bed for more complicated (and more physically relevant) quantum field theories. We stress that from this point of view, one of the directions of space becomes the time variable in the quantum field theory. The time variable appearing in \eqref{e:eqX} is then seen as an additional, physically fictitious variable, sometimes called the ``stochastic time'' in the literature. The construction of a quantum field theory based on the invariant measure of a random process is called ``stochastic quantisation'', and was proposed by Parisi and Wu \cite{ParisiWu}. We refer to \cite[Section~20.1]{GlimmJaffeBook} for references and more precise explanations.

We now return to the problem that \eqref{e:eqX} and \eqref{QFT-measure} do not actually make sense mathematically. In \eqref{e:eqX}, the roughness of the noise requires $X$ to be distribution-valued, and therefore the interpretation of the cubic power is unclear. In \eqref{QFT-measure}, one could interpret 
\begin{equation*}  
\exp\Ll(- \int_{\T^3}  | \nabla X|^2 \Rr) \prod_{x \in \T^3} \mathrm d X(x)
\end{equation*}
as a formal notation to denote the law of a Gaussian free field. Again, the Gaussian free field is distribution-valued, and there is no canonical interpretation for $X^4$ or $X^2$. 

From now on, we focus on making sense of \eqref{e:eqX}. A naive attempt would consist of regularising the noise, defining the corresponding solution, and trying to pass to the limit. However, the progressive blow-up of the non-linearity forces the limit to be identically zero (see \cite{triviality_phi4} for a rigorous justification in the case of two space dimensions). Thus, we need to take a step back and modify the original equation \eqref{e:eqX} in a way that is faithful to the intended ``physics" of the model, as sketched above. 

A formal scaling argument (see e.g.\ \cite[Section~1.1]{LectureNotesAjay}) shows that the non-linearity should become less and less relevant as we zoom in on the solution: the equation is said to be \emph{subcritical}, or super-renormalisable. The basic idea for making sense of the equation is therefore to postulate a first-order expansion of $X$ of the form $X = \<1> + Y$, where $\<1>$ is the stationary solution to the linear equation
\begin{equation}
\label{SHE}
(\partial_t  -\Delta + 1 )\<1> = \xi.
\end{equation}
In other words, letting $\{P_t = e^{t(\Delta - 1)}\}_{t \ge 0}$ denote the heat semigroup on $\T^3$, we have
\begin{equation}  \label{e.yui}
\<1>(t) = \int_{-\infty}^t P_{t-s} (\xi(s)) \, \d s.
\end{equation}
The ``$+1$'' in \eqref{SHE} serves to prevent the divergence of the low-frequency part of $\<1>$ in the long-time limit (and to allow us to talk about a stationary solution over the whole time line $\R$).

Making the ansatz $X = \<1> + Y$ and formally rewriting \eqref{e:eqX} in terms of $Y$ leads to the equation
\begin{equation}\label{Y-equation}
\partial_t Y  = \Delta Y - (Y + \<1>)^3  + m (Y+\<1>).
\end{equation}
Solving this equation for $Y$ requires  us to make sense of quantities such as $(\<1>)^2$ or $(\<1>)^3$. 

These are again ill-defined. We may regularise the noise, on scale $1/n$, and define the corresponding solution~$\<1>_n$. While $(\<1>_n)^2$ and $(\<1>_n)^3$ still diverge as we remove the regularisation, the very explicit and simple structure of $\<1>$ allows us now to identify a constant $\cc_n$  such that $(\<1>_n)^2 - \cc_n$ and $(\<1>_n)^3 - 3 \cc_n \, \<1>_n$ converge to non-trivial limits as $n$ tends to infinity, which we denote by $\<2>$ and $\<3>$ respectively. 

At this point, we can rewrite the equation \eqref{Y-equation} for $Y = X - \<1>$ using $\<1>$, $\<2>$ and~$\<3>$. In two space dimensions, this equation has been solved in \cite{dPD} with classical methods, without further recourse to the probabilistic structure of the problem. Note that this  approach shares the philosophy of rough path theory (see e.g.\ \cite{FrizHairer} for an introduction), in that one first constructs a few fundamental objects (here $\<1>$, $\<2>$ and $\<3>$) by relying on the probabilistic structure of the problem, and then one builds the solution as a deterministic and continuous map of the enriched datum $(\<1>,\<2>, \<3>)$. 

In three space dimensions, the equation one obtains for the remainder $Y$ is still ill-defined, and we need to proceed further in the postulated ``Taylor expansion'' of the solution $X$. The procedure becomes more intricate, and was solved only recently by Hairer \cite{Martin1} within his ground-breaking theory of \emph{regularity structures} (see also \cite{MartinKPZ} for the treatment of the KPZ equation with rough paths). Catellier and Chouk \cite{catcho} then showed how to recover the results of Hairer for the $\Phi^4_3$ model, using the alternative theory of \emph{para-controlled distributions} set up by Gubinelli, Imkeller and Perkowski \cite{Gubi}. We refer to \cite[Section~1]{global} for a presentation of the latter approach with notation consistent with the one we use here. Yet another approach based on Wilsonian \emph{renormalisation group} analysis was given by Kupiainen in \cite{kuppi}.

We work here in the para-controlled framework, as in 
\cite{Gubi,catcho,global}. As it turns out, six ``basis'' elements, that is, non-linear objects based on the solution to the linear equation \eqref{SHE} and built using the probabilistic structure, are required to define a solution to the $\Phi^4_3$ equation. We call these processes ``the diagrams''. They are listed in Table~\ref{t:diag}; their precise meaning will be explained shortly. The purpose of this note is to review their construction. 

Minor variants of these diagrams were built in \cite{Martin1} in the context of regularity structures. There, a very convenient graphical notation akin to Feynman diagrams was introduced to derive the bounds required for the construction of these diagrams (see \cite{Polyak} for an elementary introduction to Feynman diagrams. An earlier version of a graph-based method to bound stochastic terms using diagrams in the context of the KPZ equation was developed in   \cite{MartinKPZ}). In the context of para-controlled distributions, the exact same diagrams as those we consider here were also built in \cite{catcho}, albeit with a possibly less transparent notation. 
More recently, a remarkable machinery was developped in the context of regularity structures, which ensures the convergence of diagrams for a large class of models under extremely general assumptions; see \cite[Theorem~A.3]{HairerQuastel} and \cite{Ajay}. 

This note is mostly expository: we aim to provide a gentle introduction to this graphical notation, and to make clear that it applies equally well in the para-controlled setting. We do not strive to capture the deep results in \cite{HairerQuastel,Ajay}, but rather to give a ``pedestrian'' exposition of the calculations involved. 
{\small
\begin{table}
\centering
\renewcommand{\arraystretch}{1.5}
\begin{tabular}{ccccccc}
\toprule
$\tau$ & $\<1>$& $\<2>$& $\<30>$& $\<31p>$& $\<32p>$& $\<22p>$ 
\\
\midrule
$ \ |\tau| \ $ & $\ -\frac 1 2 - \eps \ $ & $ \ -1 - \eps \ $ & $ \ \frac 1 2 - \eps \ $ & $ \ -\eps \ $ & $ \ -\frac 1 2 - \eps \ $ & $ \ -\eps \ $
\\
\bottomrule
\end{tabular}
\bigskip
\caption{The list of relevant diagrams, together with their regularity exponent, where $\eps > 0$ is arbitrary. }
\label{t:diag}
\end{table}
}

\smallskip

We now introduce some notation. Let $\mcl S'$ denote the space of Schwartz distributions over the torus $\T^3$, and define
\begin{equation}  
\label{e.def.I}
I(f) :
\Ll\{
\begin{array}{rcl}
\R & \to & \mcl S' \\
t & \mapsto & \int_{-\infty}^t P_{t-s} (f(s)) \, \d s,
\end{array}
\Rr.
\end{equation}
for every $f \in C(\R,\mcl S')$ for which the integral makes sense. 
In other words, $I(f)$ is the ``ancient'' solution to the heat equation with right-hand side $f$, that is, the one ``started at time $-\infty$''. We measure the regularity of distributions on $\T^3$ via a scale of function spaces which we denote by $\Cc^{\al}$, $\al \in \R$. The precise definition is recalled below, and is a natural extension of the notion of $\al$-H\"older regular functions. We also recall below the definition of the resonant product $\pe$. Our goal is to identify suitable deterministic constants $\cc_n, \cc_n' \in \R$ and show the convergence as $n$ tends to infinity of the following five processes:
\begin{equation}  \label{eq:processes}
	\begin{split}
\<2>_n & := (\<1>_n)^2 - \cc_n, \\
\<30>_n & := I \Ll( (\<1>_n)^3 - 3 \cc_n \<1> \Rr), \\
\<31p>_n & := \<30>_n \pe \<1>_n, \\
\<22p>_n & := I \Ll( \<2>_n \Rr) \pe \<2>_n  - 2 \cc_n', \\
\<32p>_n & := \<30>_n \pe \<2>_n- 6 \cc_n' \<1>_n. 
    \end{split}
\end{equation}
The interested reader is referred to the discussion in \cite[Section~1.1]{global} to see how
these diagrams arise in the construction of solutions to \eqref{e:eqX}.
The stationarity in space and time of the white noise $\xi$ as well as the fact that $I$ 
defines ancient solutions to the inhomogeneous heat equation imply that these processes are stationary in space and time. Here is the main result on which we will focus.
\begin{theorem}[\cite{Martin1,catcho}] \label{th:main}
Fix 
\begin{equation}
\label{e.def.ccn}
\cc_n := \E \Ll[ (\<1>_n(t))^2 \Rr] \quad \text{ and } \quad  \cc_n' := \E \Ll[ I \Ll( \<2>_n \Rr) \pe \<2>_n (t) \Rr] .
\end{equation}
For each pair $(\tau,|\tau|)$ as in Table~\ref{t:diag}, let $\tau_n$ be defined as in \eqref{eq:processes}. There exists a stochastic process, denoted by~$\tau$ and taking values in $C(\R,\Ca)$, such that for every $p \in [1,+\infty)$, we have
\begin{equation}
\label{e.conv}
\sup_{t \in \R} \E \Ll[ \|\tau_n(t) - \tau(t)\|_{\Ca}^p \Rr] \xrightarrow[n \to +\infty]{} 0.
\end{equation}
Moreover, for every $p \in [1, +\infty)$,
\begin{equation}
\label{e.bound1}
\sup_{t \in \R} \E \Ll[ \|\tau(t)\|_{\Ca}^p \Rr]  < +\infty,
\end{equation}
and for every $p \in [1,\infty)$  and $\lambda \in [0,1]$,
\begin{equation}
\label{e.bound2}
\sup_{s < t} \frac {\E\Ll[ \|\tau(t) - \tau(s)\|_{\Cc^{|\tau| - 2\lambda}}^p \Rr] }{|t-s|^\lambda} < +\infty.
\end{equation}
\end{theorem}
\begin{remark}  
Since the processes we consider are stationary in time, the constants in \eqref{e.def.ccn} do not depend on the time $t$. Moreover, the suprema in \eqref{e.conv} and \eqref{e.bound1} are superfluous. We prefer to write them nonetheless, since the statements with the suprema are robust to perturbations of the stationarity property.
\end{remark}
\begin{remark}
As will be seen below, the constants $\cc_n$ and $\cc_{n}'$ diverge at order $n$ and $\log n$ respectively. 
\end{remark}

\begin{remark}
The bound \eqref{e.bound2} immediately implies the pathwise H\"older 
continuity of the symbols, by the Kolmogorov continuity test. 
In the construction of solutions to \eqref{e:eqX}, this strong control 
on the H\"older regularity of $\tau$ is only needed for the symbol $\<30>$. For the 
remaining symbols, a weaker bound of the type
\begin{align*}
\E  \Big[ \sup_{t \in [0,T]} \|\tau(t)\|_{\Ca}^p \Big]  
\end{align*}
suffices. However, the proofs of \eqref{e.bound2} and \eqref{e.bound1} are relatively similar anyway, as we hope to convince the reader below.
\end{remark}


\smallskip

This note is organised as follows. In Section \ref{sec:pl}, we introduce  Besov spaces. The diagrams take values in these spaces. The definition of these function spaces is based on the Littlewood-Paley decomposition. This allows us to define paraproducts along the way, and to give relevant intuition for them. In Section~\ref{sec:Gaussian}, we introduce the white noise process and discuss the property of equivalence of moments. The latter is very convenient to reduce the bounds \eqref{e.bound1} and \eqref{e.bound2} to easy-to-check second moment computations. In Section~\ref{sec:bounds}, we construct the diagrams and prove the fixed time bound \eqref{e.bound1}. In Section~\ref{sec:time}, we briefly discuss how the bound \eqref{e.bound2} for time differences follows easily from the fixed time one. Finally, in the appendix, we give an alternative proof of the equivalence of moments property exposed in Section~\ref{sec:Gaussian}.

\section{Function spaces and paraproducts} \label{sec:pl}

In this section, we introduce the function spaces we will use, denoted by $\Cc^{\al}$, for $\al \in \R$. When $\al \in (0,1)$, they are (a separable version of) the usual H\"{o}lder spaces. For general $\al$, they belong to the larger class of Besov spaces, and enjoy remarkable stability properties under multiplication. We choose to also give an informal presentation of these properties, although we will not refer to these in our actual construction of the diagrams. Our choice is motivated by the fact that the question of defining products of distributions is central to making sense of the $\Phi^4$ model. It is therefore useful to survey first what can be achieved with purely deterministic methods (and what can be ultimately used to show well-posedness of the $\Phi^4$ model). Moreover, exploring this question naturally leads to the introduction of paraproducts. In this section, the space dimension $d$ is arbitrary. For most results, we only provide a sketch of proof. A much more detailed treatment of the topics discussed in this section can be found in \cite[Chapter~2]{BCD}.

We wish to extend the notion of $\al$-H\"older regularity of a distribution $f$ on $\T^d$ to exponents $\al < 0$. Roughly speaking, this should mean that for every $x \in \T^d$ and every test function $\ph \in C^\infty_c(\R^d)$, 
\begin{equation}  
\label{e.intuit.Holder}
\Ll\langle f, \eps^{-d} \ph(\eps^{-1} (\cdot \, - x))  \Rr\rangle \ls \eps^\al \quad \text{  uniformly in $x \in \T^d$ and $\eps \to 0$}, 
\end{equation}
where we interpret $f$ as a periodic distribution on $\R^d$ in the duality pairing above\footnote{Here and below we write $A \ls B$ to mean that there exists a constant $C$, which is independent of the quantities of interest, such that $A \leq CB$.}.
 A precise definition can be built along these lines (the interested reader can find it in \cite[Definition~3.7]{Martin1}). However, we prefer to adopt here a point of view based on Fourier analysis, which allows for a more direct understanding of the stability of the spaces under multiplication.

\begin{remark}
For \emph{positive} $\alpha$, the condition \eqref{e.intuit.Holder} should be replaced by
\begin{equation}  
\label{e.intuit.Holder-bis}
\Ll\langle f - p_x(\cdot) , \eps^{-d} \ph(\eps^{-1} (\cdot \, - x))  \Rr\rangle \ls \eps^\al \quad \text{  uniformly in $x \in \T^d$ and $\eps \to 0$}, 
\end{equation}
where $p_x$ is the Taylor approximation of order $\lfloor \alpha\rfloor$ of $f$. For negative $\alpha$,
there is no such polynomial, and this recentering is unnecessary.
\end{remark}

For every $f \in L^1(\T^d)$ and $\omega \in \Z^d$, we write
\begin{equation}  
\label{e.def.hatf}
\F f(\omega) = \hat f(\omega) := \int_{\T^d} f(x) e^{-2i \pi \omega \cdot x} \, \d x,
\end{equation}
for the Fourier coefficient of $f$ with frequency $\omega$, so that
\begin{equation*}  
f(x) = (\F^{-1}\hat f)(x) := \sum_{\omega \in \Z^d} \hat f(\omega) e^{2i\pi \omega \cdot x},
\end{equation*}
where $\F^{-1}$ denotes the inverse Fourier transform.

The definition of Besov spaces rests on a decomposition of the Fourier series of a function along dyadic annuli, an idea due to Littlewood and Paley. More precisely, we think of splitting $\hat f$ into
\begin{equation}  \label{e.Fourier.decomp}
\hat f \, \1_{B(0,1)} + \sum_{k = 0}^{+\infty} \hat f \,  \1_{B(0,2^{k+1}) \setminus B(0,2^k)},
\end{equation}
where $B(0,r) := \{\omega \in \Z^d \ : \ |\omega| < r\}$.  
In \eqref{e.Fourier.decomp}, the terms of the series associated with large $k$'s measure the fast oscillations of the function; the general Besov norm can be thought of as a weighted average of the $L^p$ norm of these summands. This type of decomposition enjoys better analytical properties if we replace the indicator functions by smoothened versions thereof. More precisely, we can find functions $\td{\chi}, \chi \in C^\infty_c(\R^d)$ both taking values in $[0,1]$, with supports
\begin{align*}
\text{Supp} \tilde{\chi} \subset B\Ll(0,\frac{4}{3}\Rr), \qquad \text{Supp} \chi \subset B\Ll(0,\frac{8}{3}\Rr) \setminus B\Ll(0,\frac{3}{4}\Rr), 
\end{align*}
and such that
\begin{equation}
\label{chi-prop3}
\td{\chi}(\zeta) + \sum_{k = 0}^{+\infty} \chi(\zeta/2^k) = 1, \quad \forall \zeta\in \R^d. 
\end{equation}
We may furthermore choose these functions to be radially symmetric. We write
\begin{equation}
\label{e:def:chik}
\chi_{-1} := \td{\chi}, \qquad \chi_k(\cdot) := \chi(\cdot/2^k) \quad k \ge 0.
\end{equation}
The supports of $\tilde{\chi}$ and $\chi$ ensure that $\chi_k$ and $\chi_{k'}$ overlap only if $|k-k'| \le 1$ (see \cite[Proposition~2.10]{BCD} for more details). For every $f \in C^\infty(\T)$ and $k \ge -1$, we let
$$
\delta_k f := \F^{-1} \Ll(\chi_k \, \hat{f}\Rr), 
$$
so that $\hat f = \sum_{k \ge -1} \chi_k \hat f$ (compare with \eqref{e.Fourier.decomp}) and $f = \sum_{k \ge -1} \delta_k f$. 
Let
\begin{equation}
\label{e:def:etak}
\eta_k = \F^{-1}(\chi_k), \qquad \eta = \eta_0.
\end{equation}
For $k \ge 0$, we have 
\begin{equation}
\label{e.approx.scale}
\eta_k \simeq 2^{dk} \eta(2^k\, \cdot\, ),
\end{equation}
up to a small error due to the fact that our phase space $\Z^d$ is discrete (since our state space $\T^d$ is compact)\footnote{One may estimate the error in \eqref{e.approx.scale} and show that it is negligible for our purpose, but the simplest way around this technical point is probably to interpret each periodic function on $\T^d$ as a periodic function on $\R^d$, and then use $L^p(\R^d)$ norms and the continuous Fourier transform throughout, so that \eqref{e.approx.scale} becomes exact. The continuous Fourier transform of any Schwartz distribution is well-defined, by duality. For a periodic $f \in L^1_\mathrm{loc}(\R^d)$, the Fourier transform is a sum of Dirac masses at every $\omega \in \Z^d$, each carrying a mass $\hat f(\omega)$ as defined in~\eqref{e.def.hatf}.}. For every $k \ge -1$, we have
\begin{equation}
\label{e.dk-convol}
\dk f = \eta_k \star f,
\end{equation}
where $\star$ denotes the convolution on the torus $\T^d$. 
In agreement with \eqref{e.intuit.Holder}, \eqref{e.approx.scale} and \eqref{e.dk-convol}, we define the $\Cc^\al$ norm by
\begin{equation}
\label{e.def.B-norm}
\| f \|_{\Cc^\al}:= \sup_{k \ge -1}2^{\al k } \| \delta_k f \|_{L^\infty}.
\end{equation}
One can check that this quantity is finite if $f \in C^\infty(\T^d)$. The space $\Cc^\al$ is the completion of $C^\infty(\T^d)$ with respect to this norm. This space can be realised as a subspace of the space of Schwartz distributions.

\begin{remark}  
\label{r.completion}
We depart slightly from the convention to define $\Cc^\al$ as the space of distributions with finite $\|\cdot\|_{\Cc^\al}$ norm. Our definition makes the space separable and allows us below to define products of distributions and functions via approximation. Moreover, one can check that if a distribution $f$ satisfies $\|f\|_{\Cc^\al} < \infty$, then $f \in \Cc^{\be}$ for every $\be < \al$.
\end{remark}

The most important property of Besov spaces for our purpose is the following multiplicative structure. 
\begin{proposition}
\label{p.mult}
Let $\al < 0 < \be$ be such that $\al + \be > 0$. The multiplication $(f,g) \mapsto fg$ extends to a continuous bilinear map from $\mcl C^\al \times \mcl C^\be$ to $\mcl C^\al$. 
\end{proposition}
The proof of this proposition rests on the decomposition
\begin{equation*}  
fg = \sum_{k < l-1} \de_k f \, \de_l g + \sum_{|k-l| \le 1} \de_k f \, \de_l g + \sum_{k > l+1} \de_k f \, \de_l g,
\end{equation*}
which we will write suggestively in the form
\begin{equation*}  
fg = f \pl g + f \pe g + f \pg g.
\end{equation*}
This is often called \emph{Bony's decomposition} into the \emph{paraproducts} $f \pl g$, $f \pg g = g \pl f$, and the \emph{resonant product} $f \pe g$. 

In order to prove Proposition~\ref{p.mult}, it suffices to show that each of these terms extends to a continuous bilinear map from $\mcl C^\al \times \mcl C^\be$ to $\mcl C^\al$. However, it is very important for our more general goal of making sense of the $\Phi^4$ model to be precise about the behaviour of each term separately. We thus simply assume that $\al < 0 < \be$  and $f \in \Cc^\al$, $g \in \Cc^\be$ to begin with (but do not yet prescribe the sign of $\al + \be$), and see whether and how we can estimate each of the terms in Bony's decomposition. We start with 
\begin{equation}  \label{e.pl}
f \pl g = \sum_{k < l-1} \de_k f \, \de_l g.
\end{equation}
In view of \eqref{e.def.B-norm}, in order to see which H\"{o}lder class $f \pl g$ belongs to, we need to estimate
\begin{equation} \label{eq:expression_pl}
\Ll\| \delta_{j} (f \pl g) \Rr\|_{L^{\infty}} = \Ll\| \sum_{l=-1}^{+\infty} \delta_{j} \big( \sum_{k=-1}^{l-2} \delta_{k} f \delta_{l} g \big)\Rr\|_{L^{\infty}}. 
\end{equation}
Recall that the Fourier transform of $\de_k f$ (resp. $\de_l g$) is supported on an annulus of both inner and outer radius of size about $2^{k}$ (resp. $2^{l}$). 
The important observation is that as we sum over $k \leq l-2$
, 
the Fourier transform of $\delta_{k} f \delta_{l} g$ is still supported in an annulus of inner and outer radius both proportional to  $2^{l}$. Hence, only a finite number 
of terms $l  \simeq j$ contribute to the outer sum on the right hand side above. Now, since
\begin{equation*}  
\|\de_k f\|_{L^\infty} \le 2^{-\al k} \|f\|_{\Cc^\al}, \qquad \|\de_l g\|_{L^\infty} \le 2^{-\be l} \|g\|_{\Cc^\al}, 
\end{equation*}
the right hand side of \eqref{eq:expression_pl} can then be bounded by
\begin{equation*}  
\| \sum_{k=-1}^{j-2} \delta_{k} f \delta_{j} g\|_{L^{\infty}} \leq \sum_{k = -1}^{j-2} 2^{-\al k - \be j} \|f\|_{\Cc^\al} \, \|g\|_{\Cc^\be} < C 2^{-(\al + \be) j} \|f\|_{\Cc^\al} \, \|g\|_{\Cc^\be},
\end{equation*}
where we used that $\al < 0$, and $C$ does not depend on $f$ or $g$. This shows that
\begin{equation*}  
\|f \pl g\|_{\Cc^{\al + \be}} \le C \|f\|_{\Cc^\al} \, \|g\|_{\Cc^\be}.
\end{equation*}
The same analysis applies for the term $g \pl f$, except that since we have $\be > 0$, we get
\begin{equation*}  
\sum_{l = -1}^{j-2} 2^{-\al j - \be l} \|f\|_{\Cc^\al} \, \|g\|_{\Cc^\be} \le C 2^{-\al j} \|f\|_{\Cc^\al} \, \|g\|_{\Cc^\be} ,
\end{equation*}
which implies that
\begin{equation*}  
\|g \pl f\|_{\Cc^{\al}} \le C \|f\|_{\Cc^\al} \, \|g\|_{\Cc^\be}.
\end{equation*}
Note that we have made no assumption on the sign of $\alpha + \beta$ so far. The term $f \pl g$ inherits essential features of the small scale behaviour of  $g$
``modulated'' by the low frequency modes of $f$.
This is in agreement with the fact that $f \pl g$ is more regular than $g \pl f$ under our hypothesis $\al < \be$.

We now turn to the resonant term $f \pe g$, which for simplicity we think of as being
\begin{equation*}  
\sum_{k = -1}^{+\infty} \de_k f \, \de_k g. 
\end{equation*}
The crucial difference in the analysis of this term, compared with the previous computations, is that the support of the Fourier transform of the summand indexed by $k$, which is the convolution of the annulus of radius about $2^k$ by itself, results in a \emph{ball} of radius $2^k$, as opposed to an \emph{annulus}. Therefore, every summand indexed by $k \ge l$ contributes the the $l$-th Littlewood-Paley block 
$\delta_l (f \pe g)$. The $L^\infty$ norm of each summand is bounded by 
\begin{equation*}  
2^{-(\al + \be) k} \|f\|_{\Cc^\al}  \|g\|_{\Cc^\be}.
\end{equation*}
If we want this to be summable over $k \geq l$, we need to assume $\al + \be > 0$. In this case, the sum is of order $2^{-(\al + \be)l}\|f\|_{\Cc^\al} \, \|g\|_{\Cc^\be}$, which suggests that
\begin{equation*}  
\|f \pe g \|_{\Cc^{\al + \be}} \le C \|f\|_{\Cc^\al} \, \|g\|_{\Cc^\be}.
\end{equation*}
These computations can all be made rigorous (see e.g.\ \cite{BCD}), and are summarised in Table~\ref{t.para}. Note that these relations are relevant in the construction of $\Phi^4$ since there, one needs to characterise the products between $f \in \Cc^{\al}$ and $g \in \Cc^{\beta}$ (though sometimes necessarily $\al +\be < 0$, and renormalisations are then needed for the resonant term).  

{\small
\begin{table}
\centering
\renewcommand{\arraystretch}{1.5}
\begin{tabular}{ccccc}
\toprule
 & $f \pl g$ & $g \pl f$ & $f \pe g$ & $fg$
\\
\midrule
 Regularity & $\al + \be$ & $\al$ & $\al + \be$ & $\al$
\\
Needs $\al + \be > 0$ & No & No & Yes & Yes 
\\
\bottomrule
\end{tabular}
\bigskip
\caption{{\small Summary of the regularity properties of paraproducts for $\al < 0 < \be$, $f \in \Cc^\al$ and $g \in \Cc^\be$.}}
\label{t.para}
\end{table}
}

We point out that the regularising effect of the heat kernel can be conveniently measured using the spaces $\mcl C^\al$. While we will not use this proposition in itself here, it is a useful guide to the intuition. In particular, the time singularity in \eqref{e.heat.flow} is integrable as long as the difference of regularity exponents is less than~$2$. In other words, the integration operator $I$ in \eqref{e.def.I} brings a gain of~$2$ units of regularity.

\begin{proposition}
\label{p.heat.flow}
If $\al \le \be \in \R$, then there exists $C < \infty$ such that for every $t > 0$, we have
\begin{equation}  
\label{e.heat.flow}
\|e^{t\Delta} f\|_{\Cc^\be} \le  C \, t^{\frac{\al-\be}{2}} \, \|f\|_{\Cc^\al}.
\end{equation}
\end{proposition}
\begin{proof}[Sketch of proof]
The Laplacian $\Delta$ is a multiplication operator in Fourier space. As a consequence, we have $\de_k(e^{t\Delta} f) = e^{t\Delta}(\de_k f)$, and since $\hat{\Delta}(\om) = -|\om|^2$, roughly speaking, we have $e^{t\Delta}(\de_k f) \simeq e^{-t 2^{2k}} \de_k f$. This suggests that
\begin{equation*}  
\|\de_k(e^{t\Delta} f)\|_{L^\infty} \le C \exp\Ll(-t 2^{2k}\Rr) \|\de_k f\|_{L^\infty} \le C \exp\Ll(-t 2^{2k}\Rr) 2^{-\al k} \|f\|_{\Cc^\al} ,
\end{equation*}
and therefore
\begin{equation*}  
2^{\be k} \|\de_k(e^{t\Delta} f)\|_{L^\infty} \le C  \Ll[ (2^{2k} t)^{\frac{\be-\al}{2}} \exp \Ll( -2^{2k} t \Rr)  \Rr]   t^{\frac{\al - \be}{2}}  \|f\|_{\Cc^\be}.
\end{equation*}
Since the term between square brackets is bounded uniformly over $k$ and $t$, the heuristic argument is complete. A rigorous proof can be derived using techniques similar to those exposed for Lemma~\ref{l.Bernstein+} below (see also e.g.\ \cite[Proposition~3.11]{JCH}).
\end{proof}

In view of \eqref{e.approx.scale}, we expect that for every $p,p' \in [1,\infty]$ such that $\frac 1 p + \frac 1 {p'} = 1$,
\begin{equation}  
\label{e.Bernstein0}
\sup_{k \ge -1} 2^{-\frac{dk}{p}} \|\eta_k\|_{L^{p'}} < \infty.
\end{equation}
Indeed, the relation \eqref{e.approx.scale}, and therefore the inequality \eqref{e.Bernstein0}, are immediate by scaling if the torus $\T^d$ is replaced by the full space $\R^d$ (and therefore the discrete Fourier series is replaced by the continuous Fourier transform). A rigorous proof of \eqref{e.Bernstein0} can be found e.g.\ in \cite[Lemma~B.1]{Ising}. By \eqref{e.dk-convol} and H\"older's inequality, we deduce the following lemma.

\begin{lemma}
\label{l.Bernstein}
Let $p \in [1,\infty]$. There exists $C < \infty$ such that
\begin{equation}
\label{e.Bernstein}
\|\dk f\|_{L^\infty} \le C 2^{\frac {dk}{p}} \|f\|_{L^p}
\end{equation}
for every $f \in C^\infty(\T^d)$ and $k \ge -1$. 
\end{lemma}

If we had chosen to define the decomposition $f = \sum_{k \ge -1} \dk f$ from the Fourier series decomposition displayed in \eqref{e.Fourier.decomp} based on indicator functions, then we would have $\dk \dk f = \dk f$, and we could therefore replace $f$ by $\dk f$ on the right side of \eqref{e.Bernstein}. With our actual definition of $\dk$, this replacement is also possible: instead of using that 
\begin{equation*}
\1_{B(0,2^{k+1}) \setminus B(0,2^k)} \, \1_{B(0,2^{k+1}) \setminus B(0,2^k)} = \1_{B(0,2^{k+1}) \setminus B(0,2^k)} ,
\end{equation*}
we choose a smooth function $\chi' \in C^\infty_c(\R^d)$ with support in an annulus and such that $\chi' \equiv 1$ on the support of $\chi$, so that
\begin{equation*}
\chi'\, \chi = \chi.
\end{equation*}
Setting $\dk' f := \F^{-1}(\chi'(\cdot/2^k) \hat f)$, the identity above translates into 
\begin{equation}
\label{e.dk'dk}
\dk' \dk f = \dk f.
\end{equation}
Next, we verify that the argument leading to Lemma~\ref{l.Bernstein} also applies if $\dk f$ is replaced by $\dk' f$ on the left side of \eqref{e.Bernstein}. Using \eqref{e.dk'dk}, we thus obtain the following lemma.

\begin{lemma}[Bernstein inequality]
\label{l.Bernstein+}
Let $p \in [1,\infty]$. There exists $C < \infty$ such that 
\begin{equation}
\label{e.Bernstein+}
\|\dk f\|_{L^\infty} \le C 2^{\frac {dk}{p}} \|\dk f\|_{L^p}
\end{equation}
for every $f \in C^\infty(\T^d)$ and $k \ge -1$. 
\end{lemma}

The following proposition uses the previous lemma to bound $p$-th moments of the H\"older norm of a random distribution in terms of the $p$-th moments of its decomposition in Fourier space.

\begin{proposition}[Boundedness criterion]
\label{p.estim.norm}
Let $\be < \al-\frac d p$. There exists $C < \infty$ such that for every random distribution $f$ on $\T^d$, we have
\begin{equation}
\label{e.estim.norm}
\E  \Ll[\|f\|_{\Cc^\be}^{p}\Rr] \le C \sup_{k \geq -1} 2^{\alpha k p} \, \E \Ll[ \|\delta_{k}f\|_{L^{p}}^{p}\Rr]. 
\end{equation}
\end{proposition}
\begin{proof}
By definition of the $\Cc^{\be}$ norm and then Lemma \ref{l.Bernstein+}, we have
\begin{align*}
\|f\|_{\Cc^{\be}}^{p} = \sup_{k \geq -1} 2^{\be k p} \|\de_{k}f\|_{L^{\infty}}^{p} \le C \sup_{k \geq -1} 2^{k(\be p + d)} \|\de_{k}f\|_{L^{p}}^{p}. 
\end{align*}
In order to take the expectation of $\|\de_{k}f\|_{L^{p}}^{p}$ directly, we enlarge the supremum on the right side above to a sum, and get
\begin{align*}
\E \|f\|_{\Cc^{\be}}^{p} \le C \sum_{k \geq -1} 2^{k(\be p + d)} \E \|\de_k f\|_{L^{p}}^{p} = C \sum_{k \geq -1} 2^{kp(\be + \frac{d}{p} - \al)} 2^{\al k p} \E\|\de_{k}f\|_{L^{p}}^{p}. 
\end{align*}
The announced estimate then follows, since $\al > \be + \frac{d}{p}$. 
\end{proof}

\begin{remark}  
With our definition of the space $\Cc^\be$ as a completion, the fact that a distribution $f$ satisfies $\|f\|_{\Cc^\be} < \infty$ does not imply that $f \in \Cc^\be$. However, in the context of Proposition~\ref{p.estim.norm}, when the right side of \eqref{e.estim.norm} is finite, we do have $f \in \Cc^\be$ with probability one. Indeed, we can always pick $\be' \in (\be,\al - \frac d p)$, deduce that $\|f\|_{\Cc^{\be'}}$ is finite with probability one, and conclude by Remark~\ref{r.completion}.
\end{remark}

\section{White noise and Nelson's estimate} \label{sec:Gaussian}

Proposition \ref{p.estim.norm} gives a criterion for determining whether a random distribution belongs to $\Cc^\be$ by looking at the $p$-th moment of its Paley-Littlewood blocks. However, it is often difficult to get sharp bounds of high moments of a random distribution. On the other hand, fortunately, the objects we encounter in the construction of $\Phi^4_3$ (and many other Gaussian models) are all built from multiplications of finitely many Gaussian random variables. These objects belong to a Wiener chaos of finite order, and we can therefore leverage on the equivalence of moments property, also often called Nelson's estimate, to deduce high-moment estimates from second-moment calculations. The purpose of this section is to present these arguments, and state the implied simpler criterion for belonging to $\Cc^\be$. In this section, we continue to work in arbitrary space dimension $d$. We only sketch some well-known arguments concerning iterated integrals and Wiener chaos, and refer the interested reader to \cite[Chapter~9]{kuo} for a more detailed exposition of these topics.

We start by introducing the space-time white noise. Formally, the space-time white noise $\xi$ is a centred Gaussian distribution on $\R \times \T^d$ with covariance 
\begin{equation}\label{e:delta}
\E \xi(t,x) \, \xi(t^{\prime}, x^{\prime}) = \delta(t - t^{\prime}) \, \delta^d(x - x^{\prime}) \;,
\end{equation}
where $\delta( \cdot) $ and $\delta^d( \cdot) $ denote Dirac delta functions over $\R$ and $\T^d$ respectively. Testing \eqref{e:delta} against a function $\ph  : \R \times \T^d \to \R$ leads us to postulate that
\begin{equation}\label{e:cov.xi}
\E \Ll[\xi(\ph)^2 \Rr] = \| \ph \|_{L^2(\R \times \T^d)}^2 \;. 
\end{equation}
\begin{definition}\label{def:whitenoise}
A space-time white noise over $\R\times \T^d$ is a family of centred Gaussian random variables $\{ \xi(\ph),  \, \ph \in L^2(\R \times \T^d)\}$ such that \eqref{e:cov.xi} holds.
\end{definition}
The existence of a space-time white noise follows from Kolmogorov's extension theorem. We prefer here to take a more constructive approach, based on Fourier analysis. 

Let $(W(\cdot,\omega))_{\omega \in \Z^d}$ be a family of complex-valued Brownian motions over $\R$. These Brownian motions are independent for 
different values of $\om$ except for the constraint $W(t, \om) = \overline{W(t,-\om)}$ (so that the white noise we are building is real-valued).  The magnitude of the variance is fixed by the condition
\begin{align*}
\E\Ll[ W(t,\om) W(t,-\om')\Rr] = 
\begin{cases}
|t| \quad &\text{if } \om = \om',\\
0 \quad & \text{otherwise.}  
\end{cases}
\end{align*}
We then set $\xi$ to be the time derivative of the cylindrical Wiener process
\begin{equation*}  
(t,x) \mapsto \sum_{\omega \in \Z^d} W(t,\omega) e^{2i\pi \omega \cdot x}.
\end{equation*}
More precisely, for every $\ph \in L^2(\R\times \T^d)$, we set
\begin{equation*}  
\xi(\ph) := \sum_{\omega \in \Z^d} \int_{t = -\infty}^{+\infty} \hat \ph(t,\omega) \, \d W(t,\omega),
\end{equation*}
where the integral is interpreted in It\^o's sense, and the notation $\hat \ph(t,\omega)$ stands for $\hat{\ph(t,\cdot)}(\omega)$. By It\^o's and Fourier's isometries, this expression is well-defined and the relation \eqref{e:cov.xi} is satisfied. Since $\xi(\ph)$ is also a centered Gaussian, this provides us with a construction of white noise. 
We use the somewhat informal notation
\begin{align*}
\xi(\ph) := \int_{\R \times \T^d} \ph (z) \, \xi(\d z) \;,
\end{align*}
although $\xi$ is almost surely not a measure. In particular, for a given $\ph \in L^2(\R \times \T^d)$, the random variable $\xi(\ph)$ is only defined outside of a set of measure zero, and a priori this set depends on the choice of $\ph$. 

As explained for example in \cite[Chapter~9]{kuo} or \cite[Section~1.1.2]{Nualart}, we can define iterated Wiener-It\^o integrals based on $\xi$. For each $k \ge 1$ and $\ph \in L^2((\R\times \T^d)^k)$, we denote the iterated integral of $\ph$ by
\begin{equation*}  
\xi^{\otimes k}(\ph) = \int_{(\R\times \T^d)^k} \ph(z_1,\ldots,z_k) \, \xi(\d z_1) \, \cdots \, \xi(\d z_k).
\end{equation*}
Denoting by $\td \ph$ the symmetrized function obtained from $\ph$:
\begin{equation}  \label{e.symmetrize}
\td \ph(z_1,\ldots,z_k) := \frac 1 {k!} \sum_{\sigma \in S_k} \ph(z_{\sigma(1)}, \ldots,z_{\sigma(k)}),
\end{equation}
where $S_k$ denotes the permutation group over $\{1,\ldots,k\}$, 
we have
\begin{equation}  
\label{e.symmetrization}
\xi^{\otimes k}(\ph) = \xi^{\otimes k} (\td \ph) = k! \int_{(\R\times \T^d)^k} \td \ph(z_1,\ldots,z_k) \, \1_{\{t_1  < \cdots < t_k \}} \, \xi(\d z_1) \, \cdots \, \xi(\d z_k),
\end{equation}
where $t_i$ is the time component of $z_i$. Moreover, we have the isometry property
\begin{equation}
\label{e.isometry}
\E \Ll[ \xi^{\otimes k}(\ph)^2 \Rr] = \E \Ll[ \xi^{\otimes k}(\td \ph)^2 \Rr] = \int_{(\R \times \T^d)^k} \td \ph^2(z_1,\ldots,z_k)  \, \d z_1 \, \cdots \, \d z_k,
\end{equation} 
and by Jensen's inequality,
\begin{equation}  \label{e.io1}
\int_{(\R \times \T^d)^k} \td \ph^2(z_1,\ldots,z_k)  \, \d z_1 \, \cdots \, \d z_k \le \int_{(\R \times \T^d)^k} \ph^2(z_1,\ldots,z_k)  \, \d z_1 \, \cdots \, \d z_k.
\end{equation}
Assuming now for notational convenience that $\ph$ is a symmetric function, that is, $\ph = \td \ph$, we may rewrite the expression \eqref{e.symmetrization} for the iterated integral $\xi^{\otimes k}(\ph)$ as a series of finite-dimensional iterated Wiener-It\^o integrals:
\begin{equation} 
\label{e.def.via.Fourier}
\xi^{\otimes k}(\ph) = k! \sum_{\om_1, \ldots, \om_k \in \Z^d} \int_{t_1 < \cdots < t_k} \hat \ph(t_1, \om_1, \ldots, t_k, \om_k) \, \d W(t_1,\om_1) \, \cdots \, \d W(t_k,\om_k),
\end{equation}
where
\begin{equation*}  
\hat \ph(t_1, \om_1, \ldots, t_k, \om_k) \\
:= \int_{(\T^d)^k} \ph(t_1,x_1,\ldots,t_k, x_k) e^{-2i\pi(\om_1 \cdot x_1 + \cdots + \om_k \cdot x_k)} \, \d x_1 \, \cdots \, \d x_k.
\end{equation*}
See also \cite[Section~9.6]{kuo} for a definition of iterated integrals.
We let
\begin{equation*}  
\mcl H_k := \Ll\{ \xi^{\otimes k}(\ph), \ \ph \in L^2((\R\times \T^d)^k) \Rr\} 
\end{equation*}
denote the $k$-th \emph{Wiener chaos}, with $\mcl H_0 = \R$. Denote by $(\Omega,\mcl F,\P)$ the probability space on which $\xi$ is defined. The spaces $\mcl H_k$ are orthogonal in $L^2(\Omega,\mcl F,\P)$. Moreover, although we will not use it, we recall that if $\mcl F$ is the $\sigma$-algebra generated by $\mcl H_1$, then
\begin{equation*}  
L^2(\Omega,\mcl F,\P) = \bigoplus_{k = 0}^{+\infty} \mcl H_k
\end{equation*}
(the interested reader may find this result in \cite[Section~9.5]{kuo}).
The more important property for our purpose is the following.
\begin{lemma}
\label{l.identif.Hn}
For each $n \in \N$, the closure in $L^2(\Omega,\mcl F,\P)$ of the linear span of the set
\begin{equation}  \label{e.lin.comb}
\Ll\{\xi(\ph_1) \, \cdots \, \xi(\ph_k), \ k \le n, \ \ph_1, \ldots, \ph_k \in L^2(\R \times \T^d) \Rr\}.
\end{equation}
coincides with
\begin{equation}  \label{e.def.Hgen}
\mcl H_{\le n} := \bigoplus_{k = 0}^n \mcl H_k.
\end{equation}
\end{lemma}
\begin{proof}[Sketch of proof]
Let $H_n = H_n(X,T)$ be the Hermite polynomials, defined recursively by 
\begin{equation}
\label{e:def-Hermite1}
\left\{
\begin{array}{l}
H_0 = 1, \\
H_{n} = X H_{n-1} - T \, \partial_X H_{n-1} \qquad ( n \in \N),
\end{array}
\right.
\end{equation}
so that $H_1 = X$, $H_2 = X^2-T$, $H_3 = X^3 - 3 XT$, etc. By a recursive application of It\^o's formula, see \cite[Theorem 9.6.9]{kuo}, we have, for every $\ph \in L^2(\R \times \T^d)$ and $n \in \N$, 
\begin{equation}
\label{e.Hermite.int}
\xi^{\otimes n}(\ph^{\otimes n}) = H_n(\xi(\ph), \|\ph\|_{L^2(\R\times \T^d)}^2).
\end{equation}
This relation already shows that every $n$-fold iterated integral is a linear combination of elements of \eqref{e.lin.comb}. Conversely, it also shows that every Hermite polynomial in $\xi(\ph)$ of degree at most $n$ --- and therefore every polynomial in $\xi(\ph)$ of degree at most $n$ --- belongs to $\mcl H_{\le n}$. The full proof of Lemma~\ref{l.identif.Hn} can be derived from \cite[Theorem 9.5.4]{kuo}.
\end{proof}

By extension, we say that a stochastic process $\tau : \R \to \Ss'(\T^d)$ belongs to $\mcl H_n$ (resp.\ $\mcl H_{\le n}$) if for every $t$ and smooth test function $\phi$, we have 
\begin{equation*}  
\langle \tau(t), \phi \rangle \in \mcl H_n \ \mbox{(resp. $\mcl H_{\le n}$)}.
\end{equation*}
If $\tau(t)$ is in fact a continuous function of the space variable, this boils down to asking that $\tau(t,x) \in \mcl H_n$ (resp. $\mcl H_{\le n}$) for each $x \in \T^d$. By Lemma~\ref{l.identif.Hn}, the approximations to the diagrams we want to construct, see \eqref{eq:processes}, all belong to $\mcl H_{\le 5}$. Since Wiener chaoses are closed, this remains true of their candidate limits (see also Section~\ref{sec:bounds} for explicit representations).

Since $\delta_k f$ is linear in $f$ for every $k$, the fact that $f$ belongs to some Wiener chaos implies that $\delta_{k} f$ belongs to the Wiener chaos of the same order. Thus, in view of Proposition~\ref{p.estim.norm}, the possibility to estimate arbitrarily high moments of elements of a fixed Wiener chaos from their $L^2$ moments will be very convenient.
\begin{proposition}[Nelson's estimate]
\label{p.nelson}
For every $n \ge 1$ and $p \in [2,\infty)$, there exists a constant $C < \infty$ such that for every $X \in \mcl H_{\le n}$,
\begin{equation*}  
\E \Ll[ |X|^p \Rr]^\frac 1 p \le C \E \Ll[ X^2 \Rr]^\frac 1 2 .
\end{equation*}
If $X \in \mcl H_{n}$, then we can take $C = (p-1)^{\frac{n}{2}}$. 
\end{proposition}

We now give a proof of Proposition~\ref{p.nelson} based on the Burkholder-Davis-Gundy inequality (\cite[Section~IV.4]{RevYor}). The appendix contains an alternative argument based on the logarithmic Sobolev inequality. 
\begin{lemma}[BDG inequality]
	\label{l.bdg}
	Let $p \in (0,\infty)$. There exists $C < \infty$ such that if $(M_t)_{t \ge 0}$ is a continuous local martingale starting from $0$, then
	\begin{equation*}  
	\E \Ll[ \sup_{0 \le s \le t} |M_s|^p \Rr] \le C \E \Ll[ \langle M \rangle_t^{\frac p 2} \Rr] ,
	\end{equation*}
	where $\langle M \rangle$ denotes the quadratic variation of $M$. 
\end{lemma}

\begin{proof}[First proof of Proposition~\ref{p.nelson}]
	We show the result by induction on $n$. For $n = 0$, the space $\mcl H_0$ only contains constants, so the result is obvious. We now fix $n \ge 1$. 
	We need to verify the property for random variables of the form $\xi^{\otimes n}(\ph)$, $\ph \in L^2((\R \times \T^d)^n)$. By \eqref{e.symmetrization}, we may assume that $\ph$ is symmetric in its variables. 
%
	Let $\mcl F_t$ be the $\sigma$-algebra generated by 
	\begin{equation*}  
	\Ll\{\xi(\ph), \ \supp \ph \subset (-\infty,t] \times \T^d\Rr\}.
	\end{equation*}
	The process
	\begin{equation*}  
	M_t := \int_{(\R \times \T^d)^n} \ph(z_1,\ldots,z_n) \, \1_{\{t_1 < \cdots < t_n < t\}} \, \xi(\d z_1) \, \cdots \, \xi(\d z_n) \qquad (t \in \R)
	\end{equation*}
	is an $(\mcl F_t)$-martingale. This can be justified either by approximation of $\ph$ by elementary functions which vanish on the diagonal (see \cite[(1.10)]{Nualart}, and take the $A_i$ there of product form), or by appealing to the representation \eqref{e.def.via.Fourier}. Moreover,
	\begin{equation*}  
	\langle M \rangle_t = \int_{\R \times \T^d} \Ll( \int_{(\R \times \T^d)^{n-1}} \ph(z_1,\ldots,z_n) \, \1_{\{t_1 < \cdots < t_n < t\}} \, \xi(\d z_1) \, \cdots \, \xi(\d z_{n-1})  \Rr)^2 \, \d z_n.
	\end{equation*}
	By Minkowski's triangle inequality for the exponent $\frac p 2 \ge 1$,
	\begin{multline*}  
	\E\Ll[\langle M \rangle_t^{\frac p 2}\Rr]^\frac 2 p \\
	\le \int_{\R\times \T^d} \E \Ll[ \Ll( \int_{(\R \times \T^d)^{n-1}} \ph(z_1,\ldots,z_n) \, \1_{\{t_1 < \cdots < t_n < t\}} \, \xi(\d z_1) \, \cdots \, \xi(\d z_{n-1})  \Rr)^p  \Rr]^{\frac 2 p} \, \d z_n.
	\end{multline*}
	By symmetrizing $\varphi$ as in \eqref{e.symmetrize}, the induction hypothesis and \eqref{e.isometry}, we infer that
	\begin{equation*}  
	\E\Ll[\langle M \rangle_t^{\frac p 2}\Rr]^\frac 2 p 
	\le C  \int_{\R\times \T^d}  \int_{(\R \times \T^d)^{n-1}} \ph^2(z_1,\ldots,z_n) \, \1_{\{t_1 < \cdots < t_n < t\}} \, \d z_1 \, \cdots \, \d z_{n-1}   \, \d z_n.
	\end{equation*}
	The conclusion then follows from the Lemma~\ref{l.bdg}, used with $t = +\infty$.
\end{proof}
\begin{remark}  
	As hinted at in the introduction, certain Ising-type Markov processes converge (or are expected to converge) to the $\Phi^4$ model (see \cite{Ising}). The proof of Nelson's estimate based on the BDG inequality is sufficiently robust to allow for a generalization to these Markov processes. Indeed, loosely speaking, a Markov process can be thought of as an evolution equation with a random forcing that is white in time. In more precise words, a Markov process comes with a martingale structure indexed by time, and the possibly surprisingly special role played by the time variable in the proof we presented above becomes very natural in this context. Using versions of It\^o's formula and the BDG inequality for processes with jumps (\cite[Appendix~C]{Ising}), one can show that \eqref{e.Hermite.int} still holds approximately (\cite[Proposition~5.3]{Ising}) and prove a version of Nelson's estimate (\cite[Lemma~4.1]{Ising}) by following essentially the same reasoning as above. 
\end{remark}

With Nelson's estimate, we are now ready to provide a simple criterion to check the main convergence result in Theorem \ref{th:main}. Since for most of the processes defined in \eqref{eq:processes}, their limits can be characterised explicitly without referring to the limiting procedure as $n \rightarrow +\infty$ (and in the cases when the limiting procedure is necessary, it is also obvious what the limit should be), we only give detailed characterisations of the limiting processes $\tau$'s themselves. Once all the properties of the limits are well understood, the convergence does not pose any further problem.

\begin{proposition} \label{pr:criterion}
Let $n \in \N$, and let $\tau : \R \to \Ss'(\T^d)$ be a random process in $\mcl H_{\le n}$ which is stationary in space, in the sense that for every $x \in \T^d$, 
\begin{equation}  
\label{e.stationary}
\mbox{the processes $(\tau(t,\cdot))_{t \in \R}$ and $(\tau(t,x+\cdot))_{t \in \R}$ have the same law. }
\end{equation}
Let $(\htau(t,\om))_{\om \in \Z^d}$ denote the Fourier coefficients of $\tau(t)$. 
If for some $t \in \R$, there exists $C < \infty$ and $\alpha \in \R$ such that for every $\om \in \Z^d$, 
	\begin{equation} \label{eq:main_bound1}
	\E \Ll[ |\htau(t,\om)|^{2} \Rr] \le C (1 + |\om|)^{-d-2\alpha}, 
	\end{equation}
	then for every $\be < \al$, we have $\tau(t) \in \Cc^{\be}(\T^{3})$, and moreover, 
	\begin{equation} \label{eq:criterion_fixed_time}
	\E \Ll[\| \tau(t) \|_{\Cc^{\beta}}^{p}\Rr] < +\infty. 
	\end{equation}
	If, in addition to \eqref{eq:main_bound1}, there exists $\lambda \in (0,1)$ such that
	\begin{equation} \label{eq:main_bound2}
	\E \Ll[|\htau(t,\om) - \htau(s,\om)|^{2}\Rr] \le C |t-s|^{\lambda} (1 + |\om|)^{-d-2 \alpha + 2 \lambda}
	\end{equation}
uniformly in $0 < |t-s| < 1$ and $\om \in \Z^{d}$, then for every $\beta < \alpha - \lambda$, we have $\tau \in C(\R, \Cc^{\beta}(\T^{3}))$, and moreover,
	\begin{equation} \label{eq:criterion_time_difference}
	\sup_{0 < |t-s| < 1}	\frac{\E \Ll[\|\tau(t) - \tau(s)\|_{\Cc^{\beta}}^{p}\Rr]}{|t-s|^{\frac{\lambda p}{2}}} < +\infty.
	\end{equation}
\end{proposition}
\begin{proof}
\emph{Step 1.} We first show that by the stationarity assumption \eqref{e.stationary}, for every $\om, \om' \in \Z^d$,
	\begin{equation} \label{eq:independence}
\quad \om + \om' \neq 0  \quad \implies \quad  \E \Ll[ \htau(s,\om) \htau(t,\om') \Rr] = 0. 
	\end{equation}
Indeed, we have, using a slightly informal integral-sign notation,
\begin{align*}  
\E \Ll[ \htau(s,\om) \htau(t,\om') \Rr] & = \iint_{(\T^d)^2} \E[\tau(s,x) \tau(t,y)] e^{-2\pi i (\omega \cdot x + \omega' \cdot y)} \, dx \, dy \\
& = \iint_{(\T^d)^2} \E[\tau(s,x) \tau(t,y)] e^{-2\pi i \Ll[(\omega+\omega') \cdot x + \omega' \cdot (y-x)\Rr]} \, dx \, dy.
\end{align*}
By the stationarity assumption, the expectation above is a function of $(y-x)$ only. Integrating in $y$ first and then in $x$, we therefore obtain \eqref{eq:independence}.

\smallskip

\emph{Step 2.} We now focus on the proof of \eqref{eq:criterion_time_difference}; the proof of \eqref{eq:criterion_fixed_time} is only simpler.

Let $\tau_{s,t} := \tau(t) - \tau(s)$. We have
	\begin{align*}
	(\delta_{k} \tau_{s,t})(x) = \sum_{\om} \chi_{k}(\om) \htau_{s,t}(\om) e^{2 \pi i \om \cdot x}. 
	\end{align*}
	We implicitly assume here that the processes under consideration are real-valued, and therefore that for every $\om \in \Z^d$,
	\begin{equation*}  
	\htau(t,-\om) = \overline{\htau(t,\om)}.
	\end{equation*}
	Since $\chi_{k}$ is an even function, we deduce that
	\begin{align*}
	\E \Ll[|(\delta_{k} \tau_{s,t})(x)|^{2}\Rr] = \sum_{\om, \om' \in \Z^{d}} \chi_{k}(\om) \chi_{k}(\om') \E \Ll[ \htau_{s,t}(\om) \htau_{s,t}(\om') \Rr] e^{2 \pi i (\om + \om') \cdot x}. 
	\end{align*}
	Using \eqref{eq:independence} and the bound \eqref{eq:main_bound2}, we get
	\begin{align*}
	\E \Ll[|(\delta_{k} \tau_{s,t})(x)|^{2}\Rr] = \sum_{\om \in \Z^{d}} |\chi_{k}(\om)|^{2} \E \Ll[ |\htau_{s,t}(\om)|^{2} \Rr] \ls |t-s|^{\lambda} 2^{-2k (\alpha-\lambda)}.
	\end{align*}
	 The above bound holds uniformly in $k \geq -1$, $0 < |t-s| < 1$ and $x \in \T^{d}$. Also, since $\tau$ belongs to $\mcl H_{\leq n}$, Nelson's estimate implies
	\begin{align*}
	\E\Ll[ \| \delta_{k} \tau_{s,t} \|_{L^{p}}^{p}\Rr] & \le \sup_{x \in \T^d} \E \Ll[ \Ll| (\de_k \tau_{s,t})(x) \Rr|^p\Rr]  \\ 
	& \ls \sup_{x \in \T^{d}} \Ll( \E\Ll[ |(\delta_{k} \tau_{s,t})(x)|^{2} \Rr]\Rr)^{\frac{p}{2}} \\
	& \ls |t-s|^{\frac{\lambda p}{2}} 2^{-kp (\alpha-\lambda)}. 
	\end{align*}
	By Proposition \ref{p.estim.norm}, we deduce that for each $\beta < \alpha - \lambda - \frac{d}{p}$, 
	\begin{align*}
	\E \Ll[\|\tau_{s,t}\|_{\Cc^{\beta}}^{p}\Rr] \ls |t-s|^{\frac{\lambda p}{2}}, 
	\end{align*}
	uniformly over all $0 < |t-s| < 1$. 
\end{proof}

\section{Construction of the diagrams} \label{sec:bounds}

We are now ready to construct the diagrams listed in Table \ref{t:diag}, and prove the relevant bounds appearing in Theorem \ref{th:main}. We will focus on the bound \eqref{e.bound1} for fixed times, and will only briefly discuss in the next section how the continuity in time follows from there. We also omit the detailed proof of the convergence \eqref{e.conv}, since once the bounds \eqref{e.bound1} for the limit diagrams are established, the convergence of approximations follows in essentially the same way. Since all our processes $\tau$ belong to $\mcl H_{\leq 5}$ and are stationary both in space and time, we can invoke Proposition \ref{pr:criterion} and 
reduce the proof of  \eqref{e.bound1}  to showing the second moment bound \eqref{eq:main_bound1} for each $\tau$.

\smallskip
The derivation of these bounds involves the estimation of several nested integrals and sums. We use a graphical notation to represent these operations. This  has the advantage of making the manipulations with potentially very long expressions shorter and more transparent. Naturally, the price to pay is to get used to the notation. One of the aims of these notes is to convince the reader that this investment is worth their while. 

The graphical notation is heavily inspired by the treatment of the stochastic terms in \cite[Section~10]{Martin1}. One  difference is that there the calculations are performed in ``real space'', while we prefer to work with the spatial Fourier transform, and to keep the time variable fixed. It turns out that despite this change, the graphs we encounter in our approach are very similar to the ones in  \cite[Section~10]{Martin1} --- only the interpretation changes slightly. Another  difference is  that we work with resonant parts of products, while Hairer considers increments of processes. We comment on this difference below (see also \cite[Section~5]{MartinKPZ} for an earlier graphical approach to bounding stochastic quantities which are represented  using the spatial Fourier transform).

\smallskip

The presentation is separated into two parts. We first show how to represent the various processes as iterated stochastic integrals, and then derive the bounds on these.

\subsection{Iterated integral representation }\label{s:itint1}

%
%
%
%

We start by showing how all of the stochastic terms $\tau$ in Table~\ref{t:diag} can be represented as sums of iterated stochastic integrals. As stated above, each of the terms $\tau$ is an element of the inhomogeneous Wiener chaos $\mcl H_{\leq 5}$, and this sum yields the explicit decomposition into its components in the homogeneous Wiener chaoses $\mcl H_0, \ldots, \mcl H_5$. This representation as iterated integrals reduces the proof of the required moment bounds to an application of the isometry property \eqref{e.isometry}. As will be shown below, this representation also makes the choice of infinite renormalisation constants transparent.

\smallskip

\textbf{Case $\tau = \<1>$.}
We take the simplest process $\tau = \<1>$, the solution to the stochastic heat equation \eqref{SHE}, as the starting point of our discussion. For each $t \in \R$ and $\om \in \Z^3$, we can write
\begin{align}\label{stoch-representation}
\hat{\<1>}(t,\om) = \int_{u = -\infty}^t \hP_{t-u}(\om)\, \d W(u,\om) , 
\end{align}
where $(W( \cdot, \om))_\om$ is the family of complex valued Brownian motions introduced in Section~\ref{sec:Gaussian}, and, for $t \ge 0$, $\hP_{t}$ is the Fourier transform of the heat kernel for $P_t = e^{t(\Delta - 1)}$, that is,
\begin{equation} \label{eq:ft_heat_1}
\hP_{t}(\om) = e^{- t (1 + 4\pi^2|\om|^{2})} = e^{- t \scal{\om}^{2}}, 
\end{equation}
where we set 
\begin{equation*}
\scal{\om} := \sqrt{1 + 4\pi^2|\om|^{2}}
\end{equation*}
for concision.  It is convenient to extend $P_t$ and $\hat P_t$ to every time $t \in \R$, by setting
\begin{equation}  
\label{e.zeropt}
\text{for every } t < 0, \qquad P_t \equiv 0 \quad \text{and} \quad \hP_t \equiv 0,
\end{equation}
so that \eqref{stoch-representation} can be rewritten as
\begin{equation*}  
\hat{\<1>}(t,\om) = \int_{u \in \R} \hP_{t-u}(\om)\, \d W(u,\om).
\end{equation*}
In future expressions of integrals against $\d W(u,\om)$, we always understand that the variable $\om$ is fixed, and that the variable $u$ is the variable of integration. For instance, we simply write
\begin{equation}  
\label{e.stoch-representation-bis}
\hat{\<1>}(t,\om) = \int_{\R} \hP_{t-u}(\om)\, \d W(u,\om).
\end{equation}
The graphical version of \eqref{stoch-representation} or \eqref{e.stoch-representation-bis} is
\begin{equation}\label{qw2}
\hat{\<1>}(t,\om) = \,
\begin{tikzpicture}[scale=0.7,baseline=0.2cm]
\node at (0,0) [root] (below){};
\node at (0,1) [circ] (above) {};
\node at (0,-0.4) [] {\scriptsize $(t,\om)$}; 
\draw[kernel] (below) to (above); 
\end{tikzpicture}
\ .
\end{equation}
 Here the root \tikz[baseline=-3] \node [root] {}; represents the pair $(t,\om)$, i.e., the time and frequency at which we seek to evaluate $\hat{\<1>}$.
The leaf  \tikz\node [circ] {}; represents an instance of the noise $dW(u,\om)$, and the line connecting them is the kernel $\hP_{t-u}$. The time variable $u$ associated to the node \tikz\node [circ] {}; is integrated out.

\smallskip

\textbf{Case $\tau = \<2>$.}
We now proceed to represent the process $\tau = \<2>$, the limit of  $\<2>_n  := (\<1>_n)^2 - \cc_n$. We start with the product~$\<1>_n^2$, writing
\begin{align}
\notag
\hat { {\<1>}_n^2} (t,\om)& = \sum_{\substack{\om_1 +\om_2 = \om\\ |\om_i| \leq n}}  \hat{\<1>}(t,\om_1) \hat{\<1>}(t,\om_2) \\
\notag
&= \sum_{\substack{\om_1 +\om_2 = \om\\ |\om_i| \leq n}}  \Big(  \int_{-\infty}^t \hP_{t-u_1}(\om_1)\, \d W(u_1,\om_1) \Big) \Big(  \int_{-\infty}^t \hP_{t-u_2}(\om_2)\, \d W(u_2,\om_2) \Big) \\
\notag
& = \sum_{\substack{\om_1 +\om_2 = \om\\ |\om_i| \leq n}}  \Big( 2   \int_{-\infty}^t  \Big[\int_{-\infty}^{u_1} \hP_{t-u_1}(\om_1) \hP_{t-u_2}(\om_2)\, \d W(u_2,\om_2) \Big]\, \d W(u_1,\om_1)  \\
\label{e:ItInt1}
& \qquad + \1_{\{\om_1 = - \om_2\}}  \int_{-\infty}^{t} \hP_{t-u}(\om_1)\hP_{t-u}(\om_2) \, \d u  \Big),
\end{align}
where the last equality follows from It\^o's formula.  The last term on the right side vanishes for $\om \neq 0$, because in this case the conditions $\om_1 +\om_2 = \om$ and $\om_1 = -\om_2$ are incompatible. For $\om  = 0$, the sum of these terms can be rewritten as
\begin{align*}
\sum_{|\om_1| \leq n}  \int_{-\infty}^{t} |\hP_{t-u}(\om_1)|^2  \, \d u  = \sum_{|\om_1| \leq n}  \frac{1}{2 \scal{\om_1}^{2}}.
\end{align*}
This is precisely the term $\cc_n := \E \Ll[ (\<1>_n(t))^2 \Rr] $, which is of the order of $n$ as $n$ goes to infinity, and is removed in the renormalisation 
procedure. Below we will show that we can pass to the limit $n \to \infty$  in the first term on the right side of \eqref{e:ItInt1}. We denote the limit by $\hat{\<2>}(t,\om)$.  It is instructive to translate the expressions back into ``real space'', and to check that in the notation  introduced in Section~\ref{sec:Gaussian}, we have 
\begin{align}
\label{aaa1}
\hat{\<2>}(t,\om)&=  \sum_{\om_1 +\om_2 = \om} 2 \Big(  \int_{-\infty}^t  \Big[\int_{-\infty}^{u_1} \hP_{t-u_1}(\om_1) \hP_{t-u_2}(\om_2)\, \d W(u_2,\om_2) \Big]\, \d W(u_1,\om_1) \Big)\\
\notag
 & = \int_{(\R \times \T^3)^2}     \int_{\T^3}P_{t-u_1}(y-x_1) P_{t-u_2}(y-x_2) e^{-i 2 \pi \om \cdot y}   \, \d y  \, \xi(\d u_1, \d x_1) \xi(\d u_2, \d x_2),
\end{align}
where we identify the operator $P_t$ with its kernel, which we  interpret as being null for $t\leq 0$. This expression shows in particular that $\hat{\<2>}(t,\om)$  is an element of the homogeneous Wiener chaos $\mcl {H}_2$ as defined in Section~\ref{sec:Gaussian}.

\smallskip

By the definition of iterated stochastic integrals, we may rewrite the identity in \eqref{e:ItInt1} as
\begin{multline*}
\hat { {\<1>}_n^2} (t,\om) =  \sum_{\substack{\om_1 +\om_2 = \om\\ |\om_i| \leq n}}  \Big(    \int_{ -\infty}^t  \int_{ -\infty}^{t} \hP_{t-u_1}(\om_1) \hP_{t-u_2}(\om_2)\, \d W(u_2,\om_2) \, \d W(u_1,\om_1)  \\
 + \mathbf{1}_{\{\om_1 = -\om_2\}}\int_{-\infty}^{t} \hP_{t-u}(\om_1) \hP_{t-u}(\om_2) \, \d u  \Big),
\end{multline*}
or, with our convention \eqref{e.zeropt},
\begin{multline}
\label{e:ItInt21bis}
\hat { {\<1>}_n^2} (t,\om) =  \sum_{\substack{\om_1 +\om_2 = \om\\ |\om_i| \leq n}}  \Big(    \int_{\R^2} \hP_{t-u_1}(\om_1) \hP_{t-u_2}(\om_2)\, \d W(u_2,\om_2) \, \d W(u_1,\om_1)  \\
 + \mathbf{1}_{\{\om_1 = -\om_2\}}  \int_{-\infty}^{t} \hP_{t-u}(\om_1) \hP_{t-u}(\om_2) \, \d u  \Big).
\end{multline}
As the reader can see, this expression is already quite bulky --- and much worse is to come. This motivates to introduce a graphical notation 
which encodes these expressions in a much more transparent way, and which we will be able to manipulate directly. At this stage, we disregard the truncation, 
 and perform all calculations in the formal limit $n = \infty$. The expression~\eqref{e:ItInt21bis} then becomes 
\begin{align*}
\widehat{\<1>^{2}}(t,\om) = \phantom{1}
\begin{tikzpicture}[scale=0.6,baseline=-0.5cm]
\node at (-1,0.3) [circ] (left) {}; 
\node at (0,-1) [root] (middle){};
\node at (1,0.3) [circ] (right) {};
\node at (0,-1.5) [] {\scriptsize $(t,\om)$}; 
\draw[kernel] (middle) to (left); 
\draw[kernel] (middle) to (right); 
\end{tikzpicture}
\phantom{1} + \phantom{1}
\begin{tikzpicture}[scale=0.6,baseline=-0.5cm]
\node at (0,-1) [root] (below) {}; 
\node at (0,0.6) [dot] (above) {};
\node at (0,-1.5) [] {\scriptsize $(t,\om)$}; 
\draw[kernel, bend left = 60] (below) to (above); 
\draw[kernel, bend right = 60] (below) to (above); 
\end{tikzpicture}.
\end{align*}
As before the root \tikz[baseline=-3] \node [root] {}; of the graph represents  the pair $(t,\om)$, and each of the leaves  \tikz\node [circ] {}; represents one occurrence of the white noise and carries a pair $(u_i, \om_i)$ itself. The kernel $\hat P_t$ is represented by the arrow connecting the nodes evaluated at time $(t-u_i)$, and the arrow points towards the node whose time variable is ``earlier''. Then all time variables $u_i$ except for the one $t$ at the node are integrated out and the $\om_i$ are summed over, subject to $\om_1 + \om_2 = \om$. We will see below that this last rule corresponds to Kirchhoff's law that the  ``ingoing'' variable $\om$ must coincide with the sum of all ``outgoing'' $\om_i$'s. The second graph on the right side is obtained by ``contracting'' the two nodes of 
the first graph.  In this second graph, we may associate a frequency $\omega_1$ to the left arrow, and a frequency $\omega_2$ to the right arrow. Kirchhoff's law for the bottom node then imposes $\omega = \omega_1 + \omega_2$, and for the top node, $\omega_1 + \omega_2 = 0$, and we recover that this term is zero unless $\omega = 0$. This contracted graph is removed in the renormalisation procedure, so that
\begin{align}\label{qw3}
\hat{\<2>}(t,\om) = \phantom{1}
\begin{tikzpicture}[scale=0.6,baseline=-0.5cm]
\node at (-1,0) [circ] (left) {}; 
\node at (0,-1) [root] (middle){};
\node at (1,0) [circ] (right) {};
\node at (0,-1.5) [] {\scriptsize $(t,\om)$}; 
\draw[kernel] (middle) to (left); 
\draw[kernel] (middle) to (right); 
\end{tikzpicture}\; . 
\end{align}

\smallskip
\textbf{Case $\tau = \<30>$.}
We now discuss the next term $\<30>$, which as announced arises as the limit of 
$\<30>_n  := I \Ll( (\<1>_n)^3 - 3 \cc_n \<1> \Rr)$. As above in \eqref{e:ItInt1} (see also \eqref{e.Hermite.int} with $n = 3$) we can use It\^o's formula to obtain an iterated integral representation for $\hat{\<1>_n^3}$,
which takes the form 
\begin{align}
\notag
\hat { {\<1>}_n^3} (t,\om)
& = \sum_{\substack{\om_1 +\om_2 + \om_3 = \om\\ |\om_i| \leq n}} 6   \int_{-\infty}^t \int_{-\infty}^{u_1}\int_{-\infty}^{u_2}  \hP_{t-u_1}(\om_1) \hP_{t-u_2}(\om_2) \hP_{t-u_3}(\om_3) \\
\notag
& \qquad \qquad \qquad \qquad \, \d W(u_3,\om_3)\, \d W(u_2,\om_2)\, \d W(u_1,\om_1)  \\
\label{e:ItInt2}
& \qquad + 3 \cc_n \hat{\<1>}_n(t,\om).
\end{align}
We had already seen above that $\cc_n$ diverges as $n$ goes to $\infty$, which motivates to remove the second term in the renormalisation procedure. The reason why we choose to work with the integrated object $\<30>$ rather than $\<3>$ is that (as we will see below) although the latter can be defined as a space time distribution, it cannot be evaluated at any fixed $t$. (This is similar to temporal white noise which can also only be interpreted when tested against a function of time and space, but never pointwise in $t$). In fact, this is an instance of the well-known fact that Wick powers of order $\geq 3$ over the three dimensional Gaussian free field do not exist (since the covariance function would not be integrable; see e.g.\ \cite[Section~2.7]{e-jentzen-shen}), and for readers familiar with this fact, it may be surprising that Wick powers up to order $4$ can be constructed as space-time objects. Using a similar graphical notation to the one used for $\<2>$, we arrive at
\begin{align*}
	\hat{\<30>}(t,\om)  = 
	\begin{tikzpicture}[scale=0.6,baseline=-0.7cm]
	\node at (0,-0.7) [dot] (middle) {}; 
	\node at (0,0.4) [circ] (up) {};
	\node at (-0.7,0) [circ] (left) {}; 
	\node at (0.7,0) [circ] (right) {}; 
	\node at (0,-1.8) [root] (below) {}; 
	\node at (0,-2.3) [] {\scriptsize $(t,\om)$}; 
	\draw[kernel] (middle) to (left); 
	\draw[kernel] (middle) to (up); 
	\draw[kernel] (middle) to (right); 
	\draw[kernel] (below) to (middle); 
	\end{tikzpicture}. 
\end{align*}

\smallskip
\textbf{Case $\tau = \<22p>$.}
At this point we want to start to think more systematically about the derivation of the diagrammatic expressions and their interpretation, and we illustrate this with the diagram $\<22p>$. This particular expansion follows from 
an iterated application of It\^o's formula, see \cite[Propositions 1.1.2 and 1.1.3]{Nualart} and \cite[Section~10]{Martin1}.

For the moment we ignore the additional complexity introduced by the resonant products $\pe$ in \eqref{eq:processes}, and give a graphical representation for $\<22>$ defined as in \eqref{eq:processes}
with $\pe$ replaced by the usual product.  To begin with, we give the graphical representation of this symbol without taking into account 
the renormalisation procedure,  i.e. we work with the random function $I(\<1>_n^2) \<1>_n^2$. 
This random function takes values in $\mcl H_{\leq 4}$, with components in $\mcl H_{4}$, $\mcl H_{2}$ and $\mcl H_{0}$.
The component in the highest Wiener chaos $\mcl H_{4}$ is represented by the graph
\begin{align}\label{GG1}
\phantom{1} \begin{tikzpicture}[scale=0.7,baseline=-0.3cm]
\node at (0,0) [dot] (middle) {}; 
\node at (-0.6,0.8) [circ] (aboveleft) {}; 
\node at (0.6,0.8) [circ] (aboveright) {}; 
\node at (0,-1.2) [root] (below){};
\node at (-0.6,-0.3) [circ] (left) {}; 
\node at (0.6,-0.3) [circ] (right) {}; 
\node at (0,-1.7) () {\scriptsize $(t,\om)$}; 
\draw[kernel] (below) to (left); 
\draw[kernel] (below) to (right); 
\draw[kernel] (below) to (middle); 
\draw[kernel] (middle) to (aboveleft); 
\draw[kernel] (middle) to (aboveright); 
\end{tikzpicture}\;,
\end{align}
i.e. precisely the graph we use as a symbol to represent this object. This graph can be interpreted as random variable either in ``real space'' coordinates or in ``Fourier coordinates''. Both interpretations are equivalent and the former is closer in spirit to  \cite[Section~10]{Martin1} and also \eqref{aaa1} above, while the latter is closer to the spirit of the present  notes. Here we present both interpretations, starting with the ``real space'' interpretation because  it is slightly easier to explain. The Fourier interpretation then follows by turning multiplication into convolution in the space coordinates:  For the ``real space'' interpretation we assign a space-time point to each of the vertices of this graph (e.g.\ $(u_1,x_1), \ldots, (u_4,x_4)$ to the four leaves, $(u_5,x_5)$ to the internal vertex, and $(t,x_6)$ to the root), an instance of the heat kernel $P$ evaluated at the difference of the variables associated to the adjacent vertices and the arrow pointing towards the ``earlier'' time variable to each  of the arrows  (e.g.\ $P_{u_5-u_1}(x_5-x_1)$ to the upper right arrow), and multiply all of these kernels. Finally, the variable corresponding to the internal variable $(u_5,x_5)$ is integrated out over space-time, the variables $(u_i,x_i)$ for the leaves are integrated against the white noise $\xi$ and we take the spatial Fourier transform with respect to the variable $x_6$ at the root, yielding the expression\footnote{Actually, for finite $n$ the heat kernels connecting to the leaves, i.e. $P_{u_5-u_1}$, $P_{u_5 - u_2}$,
$P_{t-u_3}$ and $P_{t-u_4}$ (but not $P_{t-u_5}$)
should be replaced by the regularised heat kernel $(t,x) \mapsto \sum_{|\om| \leq n} \hP_t(\om) e^{i 2 \pi \om \cdot x}$.  Similarly, in \eqref{aaa6} and after 
we will leave implicit the constraint $|\om_1|, \ldots, |\om_4| \leq n$. Here and below we drop the regularisations for convenience.}
\begin{align}
\notag
&\int_{(\R \times \T^3)^4}  \Big(  \int_{\T^3}  \d x_6 \int_{\R \times \T^3}  \d u_5  \, \d x_5 P_{u_5-u_1}(x_5-x_1) P_{u_5-u_2}(x_5-x_2) \\
\notag
& \qquad \times  P_{t-u_5}(x_6-x_5)P_{t-u_3}(x_6 - x_3) P_{t-u_4}(x_6 - x_4)  e^{-i 2 \pi \om \cdot x_6} \Big) \\
& \qquad \qquad \qquad \qquad \xi(\d u_1, \d x_1) \xi(\d u_2, \d x_2)\xi(\d u_3, \d x_3) \xi(\d u_4, \d x_4).
\label{aaa5}
\end{align}
%
Translating the previous expression and interpretation into Fourier variables can be done as follows: each of the vertices is equipped with a time-frequency variable in $\R \times \Z^3$, say $(u_1,\om_1), \ldots, (u_4, \om_4)$ for the leaves, $(u_5, \om_5)$ for the internal vertex, and $(t,\om)$ for the root. The formula \eqref{aaa5} then becomes
\begin{align}
\notag
& \sum_{\substack{\om_1, \ldots, \om_5 \in \Z^3\\ \om_1 + \om_2 = \om_5 \\ \om_3 + \om_4 + \om_5 = \om}} \int_{\R^4}  \Big(  \int_{\R } \; \d u_5  \hP_{u_5-u_1}(\om_1) \hP_{u_5-u_2}(\om_2) 
   \hP_{t-u_5}(\om_5) \hP_{t-u_3}(\om_3) \hP_{t-u_4}(\om_4) \Big) \\
& \qquad \qquad \qquad \qquad d W(u_1, \om_1)\, \d W(u_2, \om_2)\, \d W(u_3, \om_3)\, \d W(u_4, \om_4),
\label{aaa6}
\end{align}
that is, each arrow now corresponds to an instance of $\hP$, where the time variables stay the same as before, but the difference of the space variables is replaced by the frequency variable corresponding to the top of the arrow. The fact that the product turns into convolution under the Fourier transform is reflected in  the ``Kirchhoff rule''  that at each internal vertex, the sum of ``incoming'' frequency variables equals the sum of ``outgoing'' frequency variables. The same rule applies at the root, with the understanding that $\omega$ is an ``ingoing'' frequency.

The terms in the lower order Wiener chaoses $\mcl H_{2}$ and  $\mcl H_{0}$ arise from It\^o's formula, in the same spirit to our discussion of $\<2>$ above. For $\mcl H_{2}$ we get
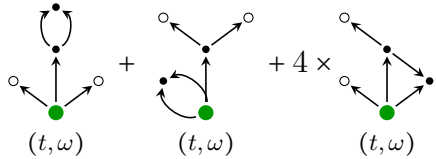
\begin{align}
\phantom{1} \begin{tikzpicture}[scale=0.7,baseline=-0.3cm]
\node at (0,0) [dot] (middle) {}; 
\node at (0,0.8) [dot] (above) {}; 
\node at (0,-1.2) [root] (below){};
\node at (-0.8,-0.6) [circ] (left) {}; 
\node at (0.8,-0.6) [circ] (right) {}; 
\node at (0,-1.8) () {\scriptsize $(t,\om)$}; 
\draw[kernel] (below) to (left); 
\draw[kernel] (below) to (right); 
\draw[kernel] (below) to (middle); 
\draw[kernel, bend left = 60] (middle) to (above); 
\draw[kernel, bend right = 60] (middle) to (above); 
\end{tikzpicture}\;
 +  
\phantom{1} \begin{tikzpicture}[scale=0.7,baseline=-0.3cm]
\node at (0,0) [dot] (middle) {}; 
\node at (-0.8,0.6) [circ] (aboveleft) {}; 
\node at (0.8,0.6) [circ] (aboveright) {}; 
\node at (0,-1.2) [root] (below){};
\node at (-0.8,-0.6) [dot] (left) {}; 
\node at (0,-1.8) () {\scriptsize $(t,\om)$}; 
\draw[kernel, bend left = 60] (below) to (left); 
\draw[kernel, bend right = 60] (below) to (left); 
\draw[kernel] (below) to (middle); 
\draw[kernel] (middle) to (aboveleft); 
\draw[kernel] (middle) to (aboveright); 
\end{tikzpicture}\;
+ 4  \times
\begin{tikzpicture}[scale=0.7,baseline=-0.3cm]
\node at (0,0) [dot] (middle) {}; 
\node at (-0.8,0.6) [circ] (aboveleft) {}; 
\node at (-0.8,-0.6) [circ] (left) {}; 
\node at (0.8,-0.6) [dot] (right) {}; 
\node at (0,-1.2) [root] (below) {}; 
\node at (0,-1.8) () {\scriptsize $(t,\om)$}; 
\draw[kernel] (middle) to (aboveleft); 
\draw[kernel] (middle) to (right); 
\draw[kernel] (below) to (left); 
\draw[kernel] (below) to (right); 
\draw[kernel] (below) to (middle); 
\end{tikzpicture}
\notag
\end{align}
These are precisely the graphs that can be obtained by picking a pair of leaves in \eqref{GG1} and ``gluing'' them together. The pre-factor ``$4$'' is combinatorial and corresponds to the fact that there are four different ways of picking one vertex from the ``top level'' and one from the ``bottom level'' of the graph, each of which giving rise to the same iterated stochastic integral. 
The interpretation for these graphs is the same as before. For instance, for the first of these graphs corresponds to the expression
\begin{align}
\notag
&  \sum_{\substack{\om_1, \ldots, \om_5 \in \Z^3\\ \om_1 + \om_2 = \om_5 \\ \om_3 + \om_4 + \om_5 = \om\\
\om_1 + \om_2 =0}} \int_{\R^2}  \Big(  \int_{\R^2 } \;  \d u_5 \, \d u_1 \hP_{u_5-u_1}(\om_1) \hP_{u_5-u_1}(\om_2) 
   \hP_{t-u_5}(\om_5) \hP_{t-u_3}(\om_3) \hP_{t-u_4}(\om_4) \Big) \\
   \notag
& \qquad \qquad \qquad \d W(u_3, \om_3)\, \d W(u_4, \om_4) \\
\notag
&= 
\cc_n \sum_{\substack{\om_3, \om_4 \in \Z^3 \\ \om_3 + \om_4  = \om}} \int_{\R^2}    \hP_{t-u_3}(\om_3) \hP_{t-u_4}(\om_4)   \, \d W(u_3, \om_3)\, \d W(u_4, \om_4) \\
\label{aa9}
& = \cc_n \hat{\<2>}(t,\om),
\end{align}
which arises from \eqref{aaa6}  by replacing the white noises $\d W(du_1, \om_1)\, \d W(du_2, \om_2)$ from the vertices which are ``glued together'' by $\delta_0(u_1 - u_2) \mathbf{1}_{\om_1 = -\om_2}$. In the same way we get the identity
\begin{align}
\label{ab1}
\begin{tikzpicture}[scale=0.7,baseline=-0.3cm]
\node at (0,0) [dot] (middle) {}; 
\node at (-0.8,0.6) [circ] (aboveleft) {}; 
\node at (0.8,0.6) [circ] (aboveright) {}; 
\node at (0,-1.2) [root] (below){};
\node at (-0.8,-0.6) [dot] (left) {}; 
\node at (0,-1.8) () {\tiny $(t,\om)$}; 
\draw[kernel, bend left = 60] (below) to (left); 
\draw[kernel, bend right = 60] (below) to (left); 
\draw[kernel] (below) to (middle); 
\draw[kernel] (middle) to (aboveleft); 
\draw[kernel] (middle) to (aboveright); 
\end{tikzpicture}\; 
= \cc_n  \times \hat{I( \<2>)} (t,\om).
\end{align}
 Finally, the term in the zero-th Wiener chaos (that is, the constant) is given by the only graph that can be obtained from two 
contractions, namely 
\begin{align}
\label{ab2}
2 \times  \begin{tikzpicture}[scale=0.7,baseline=-0.2cm]
\node at (0,0.8) [dot] (above) {}; 
\node at (0,-1.2) [root] (below) {}; 
\node at (-1,-0.2) [dot] (left) {}; 
\node at (1,-0.2) [dot] (right) {}; 
\node at (0,-1.8) () {\tiny $(t,\om)$}; 
\draw[kernel] (below) to (above); 
\draw[kernel] (below) to (left); 
\draw[kernel] (below) to (right); 
\draw[kernel] (above) to (left); 
\draw[kernel] (above) to (right); 
\end{tikzpicture} \;.
\end{align}
We now move on to discussing the renormalisation of these terms. As the reader can verify, working with 
$I ( \<2>) \<2> $ instead of  $I( \<1>^2) \<1>^2$ (i.e. performing the Wick renormalisation of the product $\<1>^2$ as above) corresponds exactly to removing the graphs in \eqref{aa9} and \eqref{ab1}. The logarithmic sub-divergence corresponding to $\cc_n'$ arises in \eqref{ab2}. Indeed, evaluating this expression  yields
\begin{align}
\notag
&   \sum_{\substack{\om_1, \ldots, \om_5 \in \Z^3\\ \om_1 + \om_2 = \om_5 \\ \om_3 + \om_4 + \om_5 = \om\\
\notag
\om_1 + \om_3 =0 \\ \om_2 + \om_4 =0}}    \int_{\R^3 } \; \d u_5 \, \d u_1 \, \d u_2  \,\hP_{u_5-u_1}(\om_1) \hP_{u_5-u_2}(\om_2) 
   \hP_{t-u_5}(\om_5) \hP_{t-u_1}(\om_3) \hP_{t-u_2}(\om_4) \\
   \notag
   = &\mathbf{1}_{ \{\om =0 \}}  \sum_{\substack{\om_1, \om_2, \om_5 \in \Z^3\\ \om_1 + \om_2 = \om_5}}  \int_{\R^3 } \;  \d u_5 \, \d u_1 \, \d u_2  \,\hP_{u_5-u_1}(\om_1) \hP_{u_5-u_2}(\om_2) \\
   \notag
   & \qquad \qquad \qquad \qquad \qquad  \times
   \hP_{t-u_5}(\om_5) \hP_{t-u_1}(-\om_1) \hP_{t-u_2}(-\om_2)\\
   \notag
 = &\mathbf{1}_{\{ \om =0 \}}  \sum_{\substack{\om_1, \om_2, \om_5 \in \Z^3\\ \om_1 + \om_2 = \om_5}}  \int_{\R } \;  \d u_5 
   \hP_{t-u_5}(\om_5)  \frac{e^{-|t-u_5| \scal{\om_1}^{2} }}{2 \scal{\om_1}^{2}}  \frac{e^{-|t-u_5| \scal{\om_2}^{2} }}{2 \scal{\om_2}^{2}}\\
\label{div.log}
   =& \mathbf{1}_{\{ \om =0 \}} \frac14 \sum_{\substack{\om_1, \om_2, \om_5 \in \Z^3\\ \om_1 + \om_2 = \om_5}} \frac{1}{  \langle \om_1 \rangle^2 }  \frac{1}{  \langle \om_2 \rangle^2 }   \frac{1}{ \langle \om_1 \rangle^2 + \langle \om_2 \rangle^2 + \langle \om_5 \rangle^2 }    ,
\end{align} 
In the first identity above, we have used the fact that the four restrictions on the~$\om_i$ are incompatible unless $\om=0$. The fact that the expression vanishes for non-zero $\om$ could also be deduced from the fact that \eqref{ab2} represents the expectation of $I(\<1>_n^2) \<1>_n^2$, which is constant in space by stationarity.
 In the second identity, we have made use of 
the symmetry of $\hat{P}$ in $\omega$ and of the fact that
$ \int_{\R}  \hP_{u_5-u_1}(\om_1) \hP_{t-u_1}(\om_1) \, \d u_1 = \frac{e^{-|t-u_5| \scal{\om_1}^{2} }}{2 \scal{\om_1}^{2}}$, as well as the corresponding identity for~$u_2$.

The sum in \eqref{div.log}, with cutoffs $|\omega_i| \le n$, diverges logarithmically as $n$ tends to infinity. It is therefore necessary to remove the diagram in \eqref{ab2} in the renormalisation procedure. We arrive at the expression 
$I(\<2>_n) \<2>_n - 2 \cc_n'$,\footnote{In fact, the sum represented by the diagram \eqref{ab2} does not coincide exactly with the constant $\cc_n'$ as defined in \eqref{e.def.ccn} because the latter is defined as the expectation of the resonant product $\E \Ll[ I \Ll( \<2>_n \Rr) \pe \<2>_n (t) \Rr]$ while the former coincides with 
$\E \Ll[ I \Ll( \<2>_n \Rr)  \<2>_n (t) \Rr]$. However, as the reader can check, the difference between these constants remains bounded as $n$ tends to infinity. } which is already very close to the definition of $\<22p>_n$ in \eqref{eq:processes}.

It only remains to re-introduce the resonant product $\pe$ in our construction, in place of the full product.  We start by briefly explaining why this is actually 
necessary, going back to the discussion of product estimates in Section~\ref{sec:pl}. In the solution theory of \eqref{e:eqX},
the term $\<22>$ plays the role of a product of $\<20> = I(\<2>)$ and $\<2>$. Now, as we have already discussed at length, $\<20>$ is a random function of 
class $\Cc^{1-}$, $\<2>$ is a random distribution of class $\Cc^{-1-}$, and products (more specifically the resonant part of products) 
are not well defined in this regularity class. (Of course, the purpose of the present article is to explain how to define these products as probabilistic limits 
of renormalised approximations). As we have just seen, a renormalisation procedure can be used to define the products. Yet, it will not improve the 
regularity of the resulting object (predicted in Table~\ref{t.para}), i.e.\ the product $\<20> \<2>$ will inherit the bad regularity $\Cc^{-1-}$ from  $\<2>$. It is however crucial, both in Hairer's theory of regularity structures and in the theory of 
paracontrolled distributions, to obtain a bound which reflects the ``good'' regularity of $\<20>$, i.e. we need to get a bound of regularity $(-1-) + (1-) = 0-$. In Hairer's theory, this is accomplished by working with ``increments'': there the fundamental object is 
\begin{align*}
(\<20>(y) - \<20>(x)) \<2>(y),
\end{align*}
which indeed behaves like $(y-x)^{0-}$ as $y \to x$. One key observation in \cite{Gubi}  was that the same effect can be obtained by 
working with the resonant product $\<20> \pe \<2>$, which is of class $\Cc^{0-}$.

After this short detour, it remains to incorporate the resonant product $\pe$ into the graphical notation. For this we recall from Section~\ref{sec:pl}
that for arbitrary $f,g$ (e.g.\ $\in \Cc^{\infty}$) we have
\begin{align*}
f \pe g(x) = \sum_{ |k -l | \leq 1} \delta_k  f(x) \delta_l g(x) = \sum_{\om_1, \om_2 \in \Z^3} e^{2i \pi (\om_1 + \om_2) \cdot x}\hat{f}(\om_1) \hat{g}(\om_2) \sum_{|k-l| \leq 1} \chi_{k}(\om_1) \chi_l(\om_2).
\end{align*}
For the last sum appearing in this expression, we have
\begin{align}\label{qu4}
 \sum_{|k-l| \leq 1} \chi_{k}(\om_1) \chi_l(\om_2) 
 \begin{cases}
 \in [0,1] \quad &\text{for all } \om_1, \om_2\\
 = 0 \quad &\text{if    } (|\om_1|> \frac{8}{3} \text{ or }  |\om_2| > \frac{8}{3}) \text{ and }   \frac{|\om_1|}{ |\om_2|} \notin [c,c^{-1}],
  \end{cases}
\end{align}
where $c= \frac{9}{64}$. Roughly speaking, and in agreement with the intuition in \eqref{e.Fourier.decomp}, this term acts as a smooth indicator function, which only selects pairs $(\om_1,\om_2)$ for which $|\om_1|$ and $|\om_2|$ are close to one another on the logarithmic scale. This intuition justifies the slightly abusive notation
\begin{align}\label{qu5}
\widehat{ f \pe g}(\om) = \sum_{\substack{\om_1 + \om_2 = \om\\ \om_1 \sim \om_2}} \hat{f}(\om_1) \hat{g}(\om_2) := \sum_{\om_1+ \om_2 = \om} \hat{f}(\om_1) \hat{g}(\om_2) \sum_{|k-l| \leq 1} \chi_{k}(\om_1) \chi_l(\om_2) .
\end{align}
We represent these ``restricted convolutions'' in the graphical notation by dotted lines. For example, when the product is replaced by $\pe$, the first graph \eqref{GG1} becomes
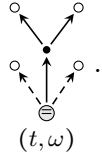
\begin{align*}
 \phantom{1} \begin{tikzpicture}[scale=0.7,baseline=-0.3cm]
\node at (0,0) [dot] (middle) {}; 
\node at (-0.6,0.8) [circ] (aboveleft) {}; 
\node at (0.6,0.8) [circ] (aboveright) {}; 
\node at (0,-1.2) [var] (below) {\tiny $=$}; 
\node at (-0.6,-0.3) [circ] (left) {}; 
\node at (0.6,-0.3) [circ] (right) {}; 
\node at (0,-1.7) () {\scriptsize $(t,\om)$}; 
\draw[kepsilon] (below) to (left); 
\draw[kepsilon] (below) to (right); 
\draw[kernel] (below) to (middle); 
\draw[kernel] (middle) to (aboveleft); 
\draw[kernel] (middle) to (aboveright); 
\end{tikzpicture} \; .
\end{align*}
The interpretation of this diagram is the same as in \eqref{aaa6}, with the only exception that the additional restriction $\{ \om_3 + \om_4 \sim \om_5 \}$ is enforced. The convention is that next to a node \begin{tikzpicture}\node at (0,0) [var]{\tiny $=$}  ; \end{tikzpicture},
the sum of the frequency variables corresponding to the dotted arrows is similar to the sum of frequency variables from the regular arrows.
Summarising all of this discussion, we finally arrive at the graphical expression
\begin{align}
\label{e.decomp.22p}
\hat{\<22p>}(t,\om) &= \phantom{1} \begin{tikzpicture}[scale=0.7,baseline=-0.3cm]
\node at (0,0) [dot] (middle) {}; 
\node at (-0.6,0.8) [circ] (aboveleft) {}; 
\node at (0.6,0.8) [circ] (aboveright) {}; 
\node at (0,-1.2) [var] (below) {\tiny $=$}; 
\node at (-0.6,-0.3) [circ] (left) {}; 
\node at (0.6,-0.3) [circ] (right) {}; 
\node at (0,-1.7) () {\scriptsize $(t,\om)$}; 
\draw[kepsilon] (below) to (left); 
\draw[kepsilon] (below) to (right); 
\draw[kernel] (below) to (middle); 
\draw[kernel] (middle) to (aboveleft); 
\draw[kernel] (middle) to (aboveright); 
\end{tikzpicture}
\phantom{1} + \phantom{1} 4 \times
\begin{tikzpicture}[scale=0.7,baseline=-0.3cm]
\node at (0,0) [dot] (middle) {}; 
\node at (-0.8,0.6) [circ] (aboveleft) {}; 
\node at (-0.8,-0.6) [circ] (left) {}; 
\node at (0.8,-0.6) [dot] (right) {}; 
\node at (0,-1.2) [var] (below) {\tiny $=$}; 
\node at (0,-1.7) () {\scriptsize $(t,\om)$}; 
\draw[kernel] (middle) to (aboveleft); 
\draw[kernel] (middle) to (right); 
\draw[kepsilon] (below) to (left); 
\draw[kepsilon] (below) to (right); 
\draw[kernel] (below) to (middle); 
\end{tikzpicture}
\phantom{1} . 
\end{align}

\smallskip

\textbf{Remaining two diagrams.}
The graphical representation/decomposition of its components in different Wiener chaoses for 
the remaining terms follows the same line of reasoning, and we omit the details. 
For $\tau = \<31p>$, we get
\begin{equation} \label{eq:expression_31first}
\hat{\<31p>}(t,\om) = \phantom{1}
\begin{tikzpicture}[scale=0.7,baseline=-0.3cm]
\node at (0,0) [dot] (middle) {}; 
\node at (0,1) [circ] (above) {};
\node at (-0.7,0.6) [circ] (left) {}; 
\node at (0.7,0.6) [circ] (right) {}; 
\node at (0,-1) [var] (below) {\tiny $=$}; 
\node at (0.8,-0.3) [circ] (middleright) {}; 
\node at (0,-1.4) [] {\tiny $(t,\om)$}; 
\draw[kernel] (middle) to (left); 
\draw[kernel] (middle) to (above); 
\draw[kernel] (middle) to (right); 
\draw[kernel] (below) to (middle); 
\draw[kepsilon] (below) to (middleright); 
\end{tikzpicture}
\phantom{1} + \phantom{1} 3 \times
\begin{tikzpicture}[scale=0.7,baseline=-0.3cm]
\node at (0,0) [dot] (middle) {}; 
\node at (-0.7,0.8) [circ] (aboveleft) {}; 
\node at (0.7,0.8) [circ] (aboveright) {}; 
\node at (0,-1.2) [var] (below) {\tiny $=$}; 
\node at (1,-0.6) [dot] (right) {}; 
\node at (0,-1.6) [] {\tiny $(t,\om)$}; 
\draw[kernel] (below) to (middle); 
\draw[kepsilon] (below) to (right); 
\draw[kernel] (middle) to (right); 
\draw[kernel] (middle) to (aboveleft); 
\draw[kernel] (middle) to (aboveright); 
\end{tikzpicture} \phantom{1}.
\end{equation}
We only mention that as above in the discussion for $\<22p>$, the Wick renormalisation (i.e. the 
fact that we work with $I(\<3>) \pe \<1>$ rather than $I(\<1>^3) \pe \<1>$) corresponds exactly
to removing the graphs in the second Wiener chaos which arise by contracting two of the three leaves
at the top of the graph. The logarithmic sub-divergence plays no role for this term.

Finally, for $\tau = \<32p>$, we have 
\begin{align*}
\hat{\<32p>}(t,\om) &= \phantom{1}
\begin{tikzpicture}[scale=0.7,baseline=0cm]
\node at (0,0) [dot] (middle) {}; 
\node at (0,1.2) [circ] (above) {}; 
\node at (-0.7,0.8) [circ] (aboveleft) {}; 
\node at (0.7,0.8) [circ] (aboveright) {}; 
\node at (0,-1) [var] (below) {\tiny $=$}; 
\node at (-0.7,-0.2) [circ] (left) {}; 
\node at (0.7,-0.2) [circ] (right) {}; 
\node at (0,-1.4) () {\tiny $(t,\om)$}; 
%
\draw[kepsilon] (below) to (left); 
\draw[kepsilon] (below) to (right); 
\draw[kernel] (below) to (middle); 
\draw[kernel] (middle) to (aboveleft); 
\draw[kernel] (middle) to (aboveright); 
\draw[kernel] (middle) to (above); 
\end{tikzpicture}
\phantom{1} + \phantom{1} 6 \times
\begin{tikzpicture}[scale=0.7,baseline=-0.2cm]
\node at (0,0) [dot] (middle) {}; 
\node at (-0.7,0.7) [circ] (aboveleft) {}; 
\node at (0.7,0.7) [circ] (aboveright) {}; 
\node at (-0.7,-0.5) [circ] (left) {}; 
\node at (0.7,-0.5) [dot] (right) {}; 
\node at (0,-1.2) [var] (below) {\tiny $=$}; 
\node at (0,-1.6) () {\tiny $(t,\om)$}; 
\draw[kernel] (middle) to (aboveleft); 
\draw[kernel] (middle) to (aboveright); 
\draw[kernel] (middle) to (right); 
\draw[kepsilon] (below) to (left); 
\draw[kepsilon] (below) to (right); 
\draw[kernel] (below) to (middle); 
\end{tikzpicture}
\phantom{1} + \phantom{1} 6 \times \bigg(
\begin{tikzpicture}[scale=0.6,baseline=0.2cm]
\node at (0,1.8) [circ] (above) {}; 
\node at (0,0.8) [dot] (middle) {}; 
\node at (0,-0.8) [var] (below) {\tiny $=$}; 
\node at (-0.8,0) [dot] (left) {}; 
\node at (0.8,0) [dot] (right) {}; 
\node at (0,-1.2) () {\tiny $(t,\om)$}; 
\draw[kernel] (middle) to (above); 
\draw[kernel] (below) to (middle); 
\draw[kepsilon] (below) to (left); 
\draw[kepsilon] (below) to (right); 
\draw[kernel] (middle) to (left); 
\draw[kernel] (middle) to (right); 
\end{tikzpicture}
\phantom{1} - \phantom{1}
\cc_n'
\cdot
\begin{tikzpicture}[scale=0.7,baseline=-0.2cm]
\node at (0,1) [circ] (above) {}; 
\node at (0,-1) [root] (below) {}; 
\node at (0,-1.4) () {\tiny $(t,\om)$}; 
\draw[kernel] (below) to (above) {}; 
\end{tikzpicture}
\bigg) .
\end{align*}
It is worth pointing out here that for this diagram, the renormalisation with the logarithmically diverging constant $\cc_n'$ does not result in the complete removal of a diagram, but only in its modification. In fact, when performing the  renormalisation procedure for more complicated equations, it is common that the renormalisation of a graph results in the subtraction of divergent substructures, rather than the removal of the whole graph. It is rather a peculiarity that in the graphs \eqref{aa9} and \eqref{ab1} the removal of the divergent substructure amounts to removing the whole graph. See \cite{Ajay} for a discussion of this point in a much more general framework.

\subsection{Bounds on iterated integrals}
We now proceed to explain how to derive the bound~\eqref{eq:main_bound1} for the various  
stochastic integrals introduced in the previous subsection. The core ingredient is the isometry 
identity \eqref{e.isometry}, which permits to bound the second moment of an iterated stochastic integral 
by the $L^2$ norm of the corresponding kernel. 

\smallskip

\textbf{The symbols $\tau = \<1>$, $\<2>$, $\<30>$.}
As before, we first treat the symbol $\<1>$. For this symbol, equation \eqref{stoch-representation} together with the standard It\^o isometry yields 
\begin{align}
\label{e.graph1.diff.times}
\E \big[  \hat{\<1>}(t, \om) \, \hat{\<1>}(t',-\om) \big] = \int_{\R}  \hP_{t-u}(\om) \hP_{t'-u}(-\om) \, \d u = \frac{e^{-|t-t'| \scal{\om}^{2} }}{2 \scal{\om}^{2}},
\end{align}
and in particular the bound~\eqref{eq:main_bound1} (for $\alpha = -\frac12$) follows from the trivial bound $e^{-|t-t'| \scal{\om}^{2}} \leq 1$. For later use, we record the following immediate corollary of the previous bound: for any $\gamma \ge 0$,
\begin{align}\label{e:qw1}
 \big| \E \big[  \hat{\<1>}(t, \om) \, \hat{\<1>}(t',-\om) \big] \big| \lesssim \frac{1}{ \scal{\om}^{2}} \Big(  \frac{1}{|t-t'| \scal{\om}^{2}} \Big)^\gamma,
\end{align}
where the implicit constant depends only the choice of $\gamma$. In graphical notation, the previous calculation with $t = t'$ becomes
\begin{equation} \label{eq:graph_free}
	\E \big[  | \hat{\<1>}(t, \om) |^2 \big] \phantom{1}  = \phantom{1}
	\begin{tikzpicture}[scale=0.6,baseline=-0.1cm]
	\node at (0,-1) [root] (below) {}; 
	\node at (0,0) [dot] (middle){};
	\node at (0,1) [root] (above) {};
	\node at (0,1.4) [] () {\tiny $(t,-\om)$}; 
	\node at (0,-1.4) [] () {\tiny $(t,\om)$}; 
	\draw[kernel] (below) to (middle); 
	\draw[kernel] (above) to (middle); 
	\end{tikzpicture}
	\phantom{1} \ls \scal{\om}^{-2},
\end{equation}
which can be obtained from \eqref{qw2} by ``doubling'' the graph and by contracting the leaves \tikz[baseline=-3] \node [circ] {}; to a vertex \tikz[baseline=-3] \node [dot] {};. The interpretation 
of the resulting graph then remains the same as in the previous section, i.e. the time variable
is integrated out. This algorithm for producing the diagram is very natural: first, taking the square corresponds to the doubling of the graph; and then taking the expectation results in collapsing each pair of leaves represented by a vertex \tikz[baseline=-3] \node [circ] {};  to a single vertex \tikz[baseline=-3] \node [dot] {};, for each possible pairing; this being due to the trivial covariance structure of the instances of white noise represented by the leaves \tikz[baseline=-3] \node [circ] {}; (in other words, this being due to It\^o's isometry). In the same ``graphical'' way, we obtain the formula 
\begin{equation} \label{eq:bound_free_square}
	\E \Ll[ |\hat{\<2>}(t,\om)|^{2} \Rr] \phantom{1} = \phantom{1} 2 \phantom{1}
	\begin{tikzpicture}[scale=0.7,baseline=-0.1cm]
	\node at (0,-1) [root] (below) {}; 
	\node at (0,1) [root] (above) {};
	\node at (-1,0) [dot] (left) {}; 
	\node at (1,0) [dot] (right) {}; 
	\node at (0,-1.4) [] {\scriptsize $(t,\om)$}; 
	\node at (0,1.4) [] {\scriptsize $(t,-\om)$}; 
	\draw[kernel] (below) to (left); 
	\draw[kernel] (below) to (right); 
	\draw[kernel] (above) to (left); 
	\draw[kernel] (above) to (right); 
	\end{tikzpicture},
	\phantom{1} 
\end{equation}
which should be read as 
\begin{align}\notag
	\E  \Ll[ |\hat{\<2>}(t,\om)|^{2} \Rr]   &= 2 \sum_{\om_1 + \om_2 = \om} \int_{\R^2} \Big( \hP_{t-u_1}(\om_1) \hP_{t-u_1} (-\om_1)  \hP_{t-u_2}(\om_2) \hP_{t-u_2} (-\om_2)  \Big) \, \d u_1 \, \d u_2 \\
\label{yu1}
	&= 2  \sum_{\om_1 + \om_2 = \om}  \frac{1}{2 \scal{\om_1}^{2}}  \frac{1}{2 \scal{\om_2}^{2}}.
\end{align}
Of course, this expression could now be bounded directly, but we prefer a slightly more general approach which will allow us to systematise the calculations to come. The following lemma, which gives a bound on discrete convolutions, is essentially contained in \cite[Lemma 10.14]{Martin1}. We formulate it in arbitrary space dimension $d$, although we are only interested in the case $d = 3$ here.

\begin{lemma}\label{l41}
Let $d \geq 1$ and $\alpha, \beta \in \R$ satisfy
\begin{align}\label{qw4}
\alpha + \beta > d \quad \text{and} \quad \alpha, \beta< d.
\end{align}
We have, uniformly over $\omega \in \Z^d$,
\begin{align}\label{ju1}
\sum_{\substack{\om_1, \om_2 \in \Z^d \\ \om_1 +\om_2 = \om}} \scal{\om_1}^{-\alpha}   \scal{\om_2}^{-\beta}  \lesssim  \scal{\om}^{d -\alpha-\beta} .
\end{align}
\end{lemma}
\begin{proof}
We subdivide the index set  $\mathcal{A} = \{ (\om_1, \om_2) \in (\Z^d)^2 \colon \om_1 +  \om_2 = \om \}$ of the summation into 
\begin{align*}
\mathcal{A}_1 &: = \big\{ (\om_1, \om_2 ) \in \mathcal{A} \colon  \quad |\om_1| \geq 2 |\om| \big\} \\
\mathcal{A}_2 &: = \big\{ (\om_1, \om_2 ) \in \mathcal{A} \colon \quad |\om_1| \leq \frac12 |\om| \big\} \\
\mathcal{A}_3 &: =\big \{ (\om_1, \om_2 ) \in \mathcal{A} \colon \quad |\om_2| \leq \frac12 |\om| \big\} \\
\mathcal{A}_4  &: = \mathcal{A} \setminus \Big( \bigcup_{j=1}^3 \mathcal{A}_j\Big) ,
\end{align*}
and bound the sums over the individual $\mathcal{A}_j$ separately. For $(\om_1, \om_2) \in \mathcal{A}_1$, we make use of the triangle inequality in the form $|\om_2| = |\om - \om_1| \geq |\om_1| - |\om| \geq \frac12 |\om_1|$ to get
\begin{align*}
\sum_{(\om_1, \om_2) \in \mathcal{A}_1} \scal{\om_1}^{-\alpha}   \scal{\om_2}^{-\beta}  \leq \sum_{(\om_1, \om_2) \in \mathcal{A}_1} \scal{\om_1}^{-\alpha}  \big(\frac12 \scal{\om_1} \big)^{-\beta}  \lesssim \sum_{|\om_1| \geq 2|\om|}
\scal{\om_1}^{-\alpha-\beta} \lesssim  \scal{\om}^{d -\alpha-\beta} ,
\end{align*}
where we have used the first condition in \eqref{qw4}. For $\mathcal{A}_2$, we use the triangle inequality in the form
$|\om_2| = |\om - \om_1| \geq |\om| - |\om_1| \geq \frac12 |\om|$ to get 
\begin{align*}
\sum_{(\om_1, \om_2) \in \mathcal{A}_2} \scal{\om_1}^{-\alpha}   \scal{\om_2}^{-\beta}  \leq \big( \frac12 \scal{\om}\big)^{-\beta} \sum_{|\om_1| \leq \frac12 |\om|} \scal{\om_1}^{-\alpha} \lesssim \scal{\om}^{d-\alpha -\beta },
\end{align*}
where this time we have used the second assumption (on $\alpha$) in \eqref{qw4}. Exchanging the role of 
$\om_1$ and $\om_2$, the same bound follows for the sum over $\mathcal{A}_3$. Finally, on $\mathcal{A}_4$ we have $ |\om_1|, |\om_2| \geq \frac12 |\om|$, so that
\begin{align*}
\sum_{(\om_1, \om_2) \in \mathcal{A}_4} \scal{\om_1}^{-\alpha}   \scal{\om_2}^{-\beta}  \leq (\frac12 \scal{\om})^{-\alpha}  \sum_{|\om_2| \ge \frac 1 2 |\om|} \big( \frac12 \scal{\om_2} \big)^{-\beta}  \lesssim \scal{\om}^{d-\alpha -\beta },
\end{align*}
and the statement follows.
\end{proof}
We briefly discuss the conditions \eqref{qw4} on the exponents $\alpha, \beta$. The first condition $\alpha + \beta>d$ is necessary to obtain a bound of the type \eqref{ju1}, because without it even the convergence of the sum cannot be guaranteed. The second condition $\alpha, \beta<d$ may seem more surprising. It states that a decay beyond summability for $\langle \om_1\rangle^{-\alpha}$ or $\langle \om_2\rangle^{-\beta}$  does not improve the behaviour of the convolution. We will see below that  this restriction  corresponds exactly to the fact that a (renormalised) product of a random function $f$ and a random distribution $g$ cannot have better regularity than $g$ itself.  We will show below how the use of the resonant product $\pe$ instead of the usual product translates into a Lemma~\ref{le:circ_convolution}, which can be understood as a variant of Lemma~\ref{l41} for which this restriction is removed.

\smallskip

Applying this Lemma to the right side of \eqref{yu1} yields
\begin{align}\notag
	\E  \Ll[ |\hat{\<2>}(t,\om)|^{2}  \Rr]  &\lesssim \frac{1}{\scal{\om}},
\end{align}
which is the desired bound \eqref{eq:main_bound1} with $\alpha=-1$ for this symbol. This implies that this process belongs to $C^{\beta}$ for every $\beta<-1$. 

\smallskip

We move on to the symbol $\<30>$. We first discuss why the need for the extra integration against the heat kernel arises. If we tried to work with $\<3>$, that is to say, to define the limit of $\<1>_n^3 - 3\cc_n \<1>_n$ (see \eqref{e:ItInt2}), then the same calculation as \eqref{eq:bound_free_square} and \eqref{yu1} would become 
\begin{align*}
\E  \Ll[ |\hat{\<3>}(t,\om)|^{2} \Rr]=6 \times  \phantom{1}
	\begin{tikzpicture}[scale=0.8,baseline=-0.2cm]
	\node at (0,0) [dot] (middle) {}; 
	\node at (-1,0) [root] (left) {};
	\node at (1,0) [root] (right) {}; 
	\node at (0,1) [dot] (above) {}; 
	\node at (0,-1) [dot] (below) {}; 
	\node at (-1.7,0) [] {\scriptsize $(t,\om)$}; 
	\node at (1.8,0) [] {\scriptsize $(t,-\om)$}; 
	\draw[kernel] (left) to (middle); 
	\draw[kernel] (left) to (above); 
	\draw[kernel] (left) to (below); 
	\draw[kernel] (right) to (middle); 
	\draw[kernel] (right) to (above); 
	\draw[kernel] (right) to (below); 
	\end{tikzpicture} \phantom{1}
	= 6  \sum_{\om_1 + \om_2 + \om_3 = \om}  \frac{1}{2 \scal{\om_1}^{2}}  \frac{1}{2 \scal{\om_2}^{2}} \frac{1}{2 \scal{\om_3}^{2}}  ,
\end{align*}
but then Lemma \ref{l41} does not apply to this situation (applying the Lemma once for the variables $\om_1$ and $\omega_2$ would yields the bound $\lesssim  \sum_{\tilde\om + \om_3 = \om}  \scal{\tilde\om}^{-1} \, \scal{\om_3}^{-2}$, but then the resulting exponents $\alpha = 1$ and $\beta =2$ just fail the summability condition $\alpha + \beta >3$.) We leave it for the reader to check that indeed, this sum diverges logarithmically for every $\om$. However, this problem can be fixed by considering different times $t \neq t'$. Then recalling \eqref{e:qw1} we get
for any $\gamma \geq 0$
\begin{align*}
\E  \Ll[ \hat{\<3>}(t,\om) \hat{\<3>}(t',-\om)\Rr] & = 6 \times  \phantom{1}
	\begin{tikzpicture}[scale=0.8,baseline=-0.2cm]
	\node at (0,0) [dot] (middle) {}; 
	\node at (-1,0) [root] (left) {};
	\node at (1,0) [root] (right) {}; 
	\node at (0,1) [dot] (above) {}; 
	\node at (0,-1) [dot] (below) {}; 
	\node at (-1.7,0) [] {\scriptsize $(t,\om)$}; 
	\node at (1.8,0) [] {\scriptsize $(t',-\om)$}; 
	\draw[kernel] (left) to (middle); 
	\draw[kernel] (left) to (above); 
	\draw[kernel] (left) to (below); 
	\draw[kernel] (right) to (middle); 
	\draw[kernel] (right) to (above); 
	\draw[kernel] (right) to (below); 
	\end{tikzpicture} \phantom{1}\\
	&	\lesssim \frac{1}{|t-t'|^{\gamma}} \sum_{\om_1 + \om_2 + \om_3 = \om}  \frac{1}{ \scal{\om_1}^{2 + 2\gamma}}  \frac{1}{ \scal{\om_2}^{2}} \frac{1}{ \scal{\om_3}^{2}}  \\
	&\lesssim \frac{1}{|t-t'|^\gamma \scal{\om}^{2\gamma} },
\end{align*}
which (for $\gamma<1$) can be taken as a basis for defining $\<3>$ as a space-time distribution. We prefer, to integrate it  once more against a heat kernel (because this is the way it enters the solution theory for \eqref{e:eqX}), yielding an object which can be evaluated at fixed time. More precisely, for every $\gamma\in (0,1)$,
\begin{align}
\notag
\E \Ll[ | \hat{\<30>}(t,\om)|^2 \Rr]&  =  6 \phantom{1}
	\begin{tikzpicture}[scale=0.8,baseline=-0.2cm]
	\node at (0,0) [dot] (middle) {}; 
	\node at (-1,0) [dot] (left) {};
	\node at (1,0) [dot] (right) {}; 
	\node at (0,1) [dot] (above) {}; 
	\node at (0,-1) [dot] (below) {}; 
	\node at (-2,0) [root] (farleft) {}; 
	\node at (2,0) [root] (farright) {}; 
	\node at (-1,-0.3) [] {\scriptsize $u$}; 
	\node at (1,-0.3) [] {\scriptsize $u'$}; 
	\draw[kernel] (left) to (middle); 
	\draw[kernel] (left) to (above); 
	\draw[kernel] (left) to (below); 
	\draw[kernel] (right) to (middle); 
	\draw[kernel] (right) to (above); 
	\draw[kernel] (right) to (below); 
	\draw[kernel] (farleft) to (left); 
	\draw[kernel] (farright) to (right); 
	\end{tikzpicture} \phantom{1}\\
	&	\lesssim \int_{\R^2} \hP_{t-u}(
	\om) \hP_{t-u'}(\om) \frac{1}{|u-u'|^\gamma \scal{\om}^{2\gamma} }  \, \d u \, \d u'
	 \lesssim \frac{1}{\scal{\om}^4},
\label{qu3}
\end{align}
which proves that this symbol satisfies \eqref{eq:main_bound1} for $\alpha=\frac{1}{2}$.

\smallskip

\textbf{Case $\tau = \<31p>$.}
We now turn to the case $\tau = \<31p>$. Recall that its decomposition into components in homogeneous Wiener chaoses was given in \eqref{eq:expression_31first}. We write
\begin{equation} \label{eq:expression_31}
\hat{\<31p>}(t,\om) = \phantom{1}
\begin{tikzpicture}[scale=0.7,baseline=-0.3cm]
\node at (0,0) [dot] (middle) {}; 
\node at (0,1) [circ] (above) {};
\node at (-0.7,0.6) [circ] (left) {}; 
\node at (0.7,0.6) [circ] (right) {}; 
\node at (0,-1) [var] (below) {\tiny $=$}; 
\node at (0.8,-0.3) [circ] (middleright) {}; 
\node at (0,-1.4) [] {\tiny $(t,\om)$}; 
\draw[kernel] (middle) to (left); 
\draw[kernel] (middle) to (above); 
\draw[kernel] (middle) to (right); 
\draw[kernel] (below) to (middle); 
\draw[kepsilon] (below) to (middleright); 
\end{tikzpicture}
\phantom{1} + \phantom{1} 3 \times
\begin{tikzpicture}[scale=0.7,baseline=-0.3cm]
\node at (0,0) [dot] (middle) {}; 
\node at (-0.7,0.8) [circ] (aboveleft) {}; 
\node at (0.7,0.8) [circ] (aboveright) {}; 
\node at (0,-1.2) [var] (below) {\tiny $=$}; 
\node at (1,-0.6) [dot] (right) {}; 
\node at (0,-1.6) [] {\tiny $(t,\om)$}; 
\draw[kernel] (below) to (middle); 
\draw[kepsilon] (below) to (right); 
\draw[kernel] (middle) to (right); 
\draw[kernel] (middle) to (aboveleft); 
\draw[kernel] (middle) to (aboveright); 
\end{tikzpicture} \phantom{1}
=: \hat{\<31p>}^{(4)}(t,\om) + \hat{\<31p>}^{(2)}(t,\om).
\end{equation}
It is then clear that $\hat{\<31p>} \in \mcl H_{n}$, and as a consequence of the orthogonality of the $\mcl H_{n}$'s in $L^{2}$, we have
\begin{equation} \label{eq:variance_31}
\E \Ll[ |\hat{\<31p>}(t,\om)|^{2}\Rr] = \E \Ll[|\hat{\<31p>}^{(4)}(t,\om)|^{2} \Rr]+ \E\Ll[ |\hat{\<31p>}^{(2)}(t,\om)|^{2}\Rr]. 
\end{equation}
In our graphical notation, the first term above can be bounded by
\begin{equation} \label{eq:31_4}
	\E \Ll[ |\hat{\<31p>}^{(4)}(t,\om)|^{2} \Rr] \lesssim   \phantom{1} 
	\begin{tikzpicture}[scale=0.8,baseline=-0.4cm]
	\node at (0,0) [dot] (middle) {}; 
	\node at (0,0.8) [dot] (above) {}; 
	\node at (0,-0.8) [dot] (below) {}; 
	\node at (-0.8,0) [dot] (left) {}; 
	\node at (0.8,0) [dot] (right) {}; 
	\node at (0,1.8) [var] (farabove) {\tiny $=$}; 
	\node at (0,-1.8) [var] (farbelow) {\tiny $=$}; 
	\node at (1.8,0) [dot] (farright) {}; 
	\node at (0,2.2) [] {\tiny $(t,-\om)$}; 
	\node at (0,-2.2) [] {\tiny $(t,\om)$}; 
	\draw[kernel] (above) to (left); 
	\draw[kernel] (above) to (middle); 
	\draw[kernel] (above) to (right); 
	\draw[kernel] (below) to (left); 
	\draw[kernel] (below) to (middle); 
	\draw[kernel] (below) to (right); 
	\draw[kernel] (farabove) to (above); 
	\draw[kepsilon] (farabove) to (farright); 
	\draw[kernel] (farbelow) to (below); 
	\draw[kepsilon] (farbelow) to (farright); 
	\end{tikzpicture}
	\phantom{1}.
\end{equation}
There is a slightly subtle point worth underlying here: standard Gaussian calculus (the Wick formula) yields an explicit identity
for the quantity on the left hand side of this expression, in terms of contractions of all of the different leaves on the original diagram representing $\hat{\<31p>}^{(4)}$. This formula includes additional graphs such as 
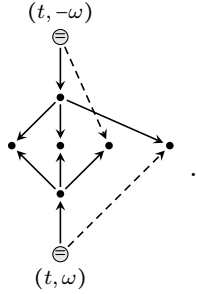
\begin{equation*}
	\begin{tikzpicture}[scale=0.8,baseline=-0.4cm]
	\node at (0,0) [dot] (middle) {}; 
	\node at (0,0.8) [dot] (above) {}; 
	\node at (0,-0.8) [dot] (below) {}; 
	\node at (-0.8,0) [dot] (left) {}; 
	\node at (0.8,0) [dot] (right) {}; 
	\node at (0,1.8) [var] (farabove) {\tiny $=$}; 
	\node at (0,-1.8) [var] (farbelow) {\tiny $=$}; 
	\node at (1.8,0) [dot] (farright) {}; 
	\node at (0,2.2) [] {\tiny $(t,-\om)$}; 
	\node at (0,-2.2) [] {\tiny $(t,\om)$}; 
	\draw[kernel] (above) to (left); 
	\draw[kernel] (above) to (middle); 
	\draw[kernel] (above) to (farright); 
	\draw[kernel] (below) to (left); 
	\draw[kernel] (below) to (middle); 
	\draw[kernel] (below) to (right); 
	\draw[kernel] (farabove) to (above); 
	\draw[kepsilon] (farabove) to (right); 
	\draw[kernel] (farbelow) to (below); 
	\draw[kepsilon] (farbelow) to (farright); 
	\end{tikzpicture}
	\phantom{1}.
\end{equation*}
However, using \eqref{e.io1} to bound the term on the left side of \eqref{eq:31_4} greatly simplifies the ensuing argument, as opposed to relying on the exact formula involving the asymmetric trees. (This idea was first used in this context in \cite[Section~10]{Martin1}.)

\smallskip

Going back to bounding \eqref{eq:31_4}, we use \eqref{qu3} on the ``left part'' and \eqref{eq:graph_free} on the ``right part'' of the graph to
obtain the bound 
\begin{align}\label{qu7}
	\E  \Ll[ |\hat{\<31p>}^{(4)}(t,\om)|^{2} \Rr] \ls 
	\sum_{\substack{\om_1+ \om_2 =\om \\ \om_1 \sim \om_2 }} \frac{1}{\langle \om_1 \rangle^4} \frac{1}{\langle \om_2 \rangle^2}.
\end{align}
Lemma~\ref{l41} on the decay of convolutions is not enough to bound the remaining sum. Indeed, this is precisely a case as discussed below Lemma~\ref{l41}, where the second condition in \eqref{qw4} fails (here, $4 \geq 3$). Of course, the same estimate could be used by ``forgetting'' some of the good decay of ${\langle \om_1 \rangle^{-4}}$, and replacing it by ${\langle \om_1 \rangle^{-(3-)}}$, but this would only yield a bound of order $\scal{\om}^{-(2-)}$ corresponding to a regularity of index $-\frac12-$ instead of $0-$. The following lemma shows that the additional condition $\om_1 \sim \om_2$, which arises from our use of the resonant product $\pe$ in the definition of the diagram, resolves this problem.

\begin{lemma} \label{le:circ_convolution}
	Let $\alpha, \beta  \in \R$ satisfy $\alpha + \beta >d$. We have, uniformly over $\om \in \Z^d$,
	\begin{equation*}
	\sum_{\substack{\om_{1} + \om_{2} = \om \\ \om_{1} \sim  \om_{2} }} \scal{\om_{1}}^{-\al} \scal{\om_{2}}^{-\be} \ls \scal{\om}^{d-\al-\be}. 
	\end{equation*}
\end{lemma}
\begin{proof}
	Recall \eqref{qu5} and the definition of the ``smooth cut-off'' \eqref{qu4}.
	For small $|\om|$ (say, $|\om| \leq \frac{16}{3})$ there is only a bounded number of admissible $\om_1, \om_2$ 	with $\om_1 + \om_2  = \om$ for which  $\sum_{|k-l| \leq 1} \chi_{k}(\om_1) \chi_l(\om_2)  \neq 0$. Hence, for such $\om$, we 	have 
$$
\sum_{\substack{\om_{1} + \om_{2} = \om \\ \om_{1} \sim  \om_{2} }} \scal{\om_{1}}^{-\al}\scal{\om_{2}}^{-\be} \ls 1.
$$ 
We can thus now assume that $|\om| > \frac{16}{3}$.
	For such $\om$, the conditions $\om_1 + \om_2 = \om$ and $\om_1 \sim \om_2$ enforce that $\frac{|\om_1|}{|\om_2|} \in [c,c^{-1}]$ for $c = \frac{9}{64}$, by \eqref{qu4}. Hence, on the one hand, we have 
	$$
	|\om| \leq |\om_1| + |\om_2| \leq |\om_1| + c^{-1} |\om_1|,
	$$ 
	that is, $|\om_1| \geq \frac{1}{1+c^{-1}} |\om|$, and on the other hand, $|\om_2| \geq c |\om_1| $. These considerations allow us to write
	\begin{equation*}
	\sum_{\substack{\om_{1} + \om_{2} = \om \\ \om_{1} \sim  \om_{2} }} \scal{\om_{1}}^{-\al} \scal{\om_{2}}^{-\be} 
	\lesssim  \sum_{|\om_1 | \geq \frac{1}{1+c^{-1}} |\om| } \scal{\om_{1}}^{-\al-\beta}  \lesssim \scal{\om}^{d - \alpha -\beta},
	\end{equation*}	
	as desired. 
\end{proof}

Applying this Lemma to \eqref{qu7} immediately yields
\begin{align*}
	\E  \Ll[ |\hat{\<31p>}^{(4)}(t,\om)|^{2} \Rr]  \ls \scal{\om}^{-3}. 
\end{align*}
We now turn to the variance of $\hat{\<31p>}^{(2)}$, which has the expression
\begin{align*}
	\E \Ll[ |\hat{\<31p>}^{(2)}(t,\om)|^{2} \Rr] \lesssim \phantom{1}
	\begin{tikzpicture}[scale=1.2,baseline=-0.2cm]
	\node at (-2.4,0) [var] (farleft) {\tiny $=$}; 
	\node at (2.4,0) [var] (farright) {\tiny $=$}; 
	\node at (0,-0.8) [dot] (belowmiddle) {}; 
	\node at (-0.8,0) [dot] (left) {}; 
	\node at (0.8,0) [dot] (right) {}; 
	\node at (-1.6,0.8) [dot] (aboveleft) {}; 
	\node at (1.6,0.8) [dot] (aboveright) {}; 
	\node at (0,0.8) [dot] (abovemiddle) {}; 
	\node at (-2.6,0) () {\tiny $\om$};
	\node at (2.7,0) () {\tiny $-\om$};
	\node at (-2.1,0.6) () {\tiny $\om_{4}$}; 
	\node at (2.1,0.6) () {\tiny $-\om_{4}'$}; 
	\node at (-1.1,0.6) () {\tiny $-\om_{4}$}; 
	\node at (1.1,0.6) () {\tiny $\om_{4}'$}; 
	\node at (-0.6,0.4) () {\tiny $\om_{1}$}; 
	\node at (0.6,0.4) () {\tiny $-\om_{1}$}; 
	\node at (-0.6,-0.4) () {\tiny $\om_{2}$}; 
	\node at (0.6,-0.4) () {\tiny $-\om_{2}$}; 
	\node at (-1.6,-.15)(){\tiny $\om_{5}$};
	\node at (1.6,-.15)(){\tiny $-\om_{5}'$};
	\draw[kepsilon] (farleft) to (aboveleft); 
	\draw[kernel] (farleft) to (left); 
	\draw[kepsilon] (farright) to (aboveright); 
	\draw[kernel] (farright) to (right); 
	\draw[kernel] (left) to (aboveleft); 
	\draw[kernel] (left) to (belowmiddle); 
	\draw[kernel] (left) to (abovemiddle); 
	\draw[kernel] (right) to (belowmiddle); 
	\draw[kernel] (right) to (aboveright); 
	\draw[kernel] (right) to (abovemiddle);
	\end{tikzpicture}. 
\end{align*}
For later reference, we have labelled all the edges with their frequency variables (but dropped the time variables) in the diagram above. 
In addition to the identities already implicit in this diagram, Kirchhoff's law enforces that we have to sum over 
the indices $\om_1, \om_2, \om_4, \om_5, \om_4', \om_5'$ satisfying
\begin{align*}
\om_4 + \om_5 = \om ,  \qquad  \om_4' + \om_5' = \om , \qquad \om_1 + \om_2 = \om, 
\end{align*}
so that we ultimately have to sum over three free variables. The fact that we have a resonant product $\pe$ at the roots yields the additional constraints $\om_4 \sim \om_5$ and $\om_4' \sim \om_5'$, but we will not rely on these additional constraints to bound this diagram.
With this notation in place, we proceed to analyse each part of this diagram separately. The inner square corresponds to the integral
\begin{align*}
&\sum_{\om_1 + \om_2 = \om} \Big(\int_\R P_{u - u_1}(\om_1) P_{u'- u_1}(-\om_1) \, \d u_1 \Big) \Big(\int_\R P_{u - u_2}(\om_2) P_{u'- u_1}(-\om_2) \, \d u_2 \Big)\\
&\lesssim \sum_{\om_1 + \om_2 = \om} \frac{1}{\scal{\om_1}^2} \frac{1}{\scal{\om_2}^2} \lesssim \frac{1}{\scal{\om}},
\end{align*}
where we have used Lemma~\ref{l41} in the last inequality.
Note that this bound is slightly sub-optimal, because we have not used the fact that the time variables $u$ and $u'$ corresponding to the nodes at the left and right corners of the square are different. In principle, this would 
yield extra factors $\exp(-|u-u'| \scal{\om_i})$ for each term, but we simply bound these factors by $1$. Similarly, we get for the left-most triangle
\begin{align*}
\sum_{\om_4 + \om_5 = \om} \int P_{t-u}(\om_5)    \Big(\int_{\R} P_{t-u_4}(\om_4) P_{u-u_4}(-\om_4) \, \d u_4 \Big) \, \d u \\
 \lesssim
 \sum_{\om_4 + \om_5 = \om} \int P_{t-u}(\om_5)   \frac{1}{\scal{\om_4}^2}  \, \d u \lesssim  \sum_{\om_4 + \om_5 = \om}  \frac{1}{\scal{\om_4}^2}   \frac{1}{\scal{\om_5}^2} \lesssim \frac{1}{\scal{\om}},  
\end{align*}
where again we have used Lemma~\ref{l41}.
The right-most triangle in the diagram is bounded in the same way, resulting in the final bound
\begin{align*}
	\E  \Ll[ |\hat{\<31p>}^{(2)}(t,\om)|^{2} \Rr]  \ls \scal{\om}^{-3}, 
\end{align*}
as desired.

\smallskip

\textbf{Case $\tau = \<22p>$.}
For this symbol, we recall from \eqref{e.decomp.22p} the Wiener chaos decomposition
\begin{align*}
\hat{\<22p>}(t,\om) &= \phantom{1} \begin{tikzpicture}[scale=0.7,baseline=-0.3cm]
\node at (0,0) [dot] (middle) {}; 
\node at (-0.6,0.8) [circ] (aboveleft) {}; 
\node at (0.6,0.8) [circ] (aboveright) {}; 
\node at (0,-1.2) [var] (below) {\tiny $=$}; 
\node at (-0.6,-0.3) [circ] (left) {}; 
\node at (0.6,-0.3) [circ] (right) {}; 
\node at (0,-1.8) () {\tiny $(t,\om)$}; 
\draw[kepsilon] (below) to (left); 
\draw[kepsilon] (below) to (right); 
\draw[kernel] (below) to (middle); 
\draw[kernel] (middle) to (aboveleft); 
\draw[kernel] (middle) to (aboveright); 
\end{tikzpicture}
\phantom{1} + \phantom{1} 4 \times
\begin{tikzpicture}[scale=0.7,baseline=-0.3cm]
\node at (0,0) [dot] (middle) {}; 
\node at (-0.8,0.6) [circ] (aboveleft) {}; 
\node at (-0.8,-0.6) [circ] (left) {}; 
\node at (0.8,-0.6) [dot] (right) {}; 
\node at (0,-1.2) [var] (below) {\tiny $=$}; 
\node at (0,-1.8) () {\tiny $(t,\om)$}; 
\draw[kernel] (middle) to (aboveleft); 
\draw[kernel] (middle) to (right); 
\draw[kepsilon] (below) to (left); 
\draw[kepsilon] (below) to (right); 
\draw[kernel] (below) to (middle); 
\end{tikzpicture}
\phantom{1} 
=: \hat{\<22p>}^{(4)}(t,\om) + \hat{\<22p>}^{(2)}(t,\om) 
\; .
\end{align*}

We start with $\hat{\<22p>}^{(4)}$. Similarly to the case for $\<31p>$,  we have the bound
\begin{equation} \label{eq:22_4}
		\E \Ll[ |\hat{\<22p>}^{(4)}(t,\om)|^{2} \Rr] \ls \phantom{1}
		\begin{tikzpicture}[scale=0.7,baseline=-0.2cm]
		\node at (0,0.8) [dot] (above) {}; 
		\node at (0,-0.8) [dot] (below) {}; 
		\node at (-0.8,0) [dot] (left) {}; 
		\node at (0.8,0) [dot] (right) {}; 
		\node at (0,1.8) [var] (farabove) {\tiny $=$}; 
		\node at (0,-1.8) [var] (farbelow) {\tiny $=$}; 
		\node at (1.8,0) [dot] (farright) {}; 
		\node at (-1.8,0) [dot] (farleft) {}; 
		\node at (0,2.2) [] {\tiny $(t,-\om)$}; 
		\node at (0,-2.2) [] {\tiny $(t,\om)$}; 
		%
		\draw[kernel] (above) to (left); 
		\draw[kernel] (above) to (right); 
		\draw[kernel] (below) to (left); 
		\draw[kernel] (below) to (right); 
		\draw[kernel] (farabove) to (above); 
		\draw[kepsilon] (farabove) to (farright); 
		\draw[kepsilon] (farabove) to (farleft); 
		\draw[kernel] (farbelow) to (below); 
		\draw[kepsilon] (farbelow) to (farright); 
		\draw[kepsilon] (farbelow) to (farleft); 
		\end{tikzpicture} 
\phantom{1}.
\end{equation}
After all our preparation, this diagram poses no additional difficulty. First, the same calculation as in \eqref{qu3} allows to bound the ``inner part'' of the diagram by ${\scal{\om_2}}^{-5}$, and the integrals corresponding to the  ``outer parts'' can be bounded by ${\scal{\om_1}^{-2}}$ and ${\scal{\om_3}^{-2}}$ immediately, yielding the bound
\begin{align}\label{qu9} 
		\E \Ll[ |\hat{\<22p>}^{(4)}(t,\om)|^{2} \Rr] \ls \sum_{\substack{\om_1 + \om_2 + \om_3 = \om\\ \om_1 + \om_3 \sim \om_2}} \frac{1}{\scal{\om_1}^2} \frac{1}{\scal{\om_2}^5}\frac{1}{\scal{\om_3}^2} \ls \sum_{\substack{\tilde\om + \om_3 = \om\\ \tilde\om \sim \om_2}} \frac{1}{\scal{\tilde \om}} \frac{1}{\scal{\om_2}^5} \ls \frac{1}{\scal{\om}^3},
\end{align}
where we have used Lemma~\ref{l41} in the first and Lemma~\ref{le:circ_convolution} in the second inequality.


\smallskip

We now turn to the bound for $\hat{\<22p>}^{(2)}$, for which we have
\begin{equation} \label{eq:22_2}
	\E \Ll[ |\hat{\<22p>}^{(2)}(t,\om)|^{2} \Rr]\ls \phantom{1}
	\begin{tikzpicture}[scale=1.1,baseline=-0.2cm]
	\node at (0,0.5) [dot] (above) {}; 
	\node at (0,-0.5) [dot] (below) {}; 
	\node at (0,1.5) [var] (farabove) {\tiny $=$}; 
	\node at (0,-1.5) [var] (farbelow) {\tiny $=$}; 
	\node at (-0.5,0) [dot] (left) {}; 
	\node at (-1.5,0) [dot] (farleft) {}; 
	\node at (1,0.5) [dot] (aboveright) {}; 
	\node at (1,-0.5) [dot] (belowright) {}; 
	\node at (0.1, 0.3) () {\tiny $u'$}; 
	\node at (0.1,-0.3) () {\tiny $u$}; 
	\node at (-0.45, 0.35) () {\tiny $-\om_{1}$}; 
	\node at (-0.45, -0.35) () {\tiny $\om_{1}$}; 
	\node at (0.5,-0.35) () {\tiny $\om_{2}$}; 
	\node at (0.5, 0.3) () {\tiny $\om_{2}'$}; 
	\node at (-1,-0.8) () {\tiny $\om_{3}$}; 
	\node at (-1,0.8) () {\tiny $-\om_{3}$}; 
	\node at (0.7, -1.2) () {\tiny $-\om_{2}$}; 
	\node at (0.7, 1.2) () {\tiny $-\om_{2}'$}; 
	\node at (0,-1.8) () {\tiny $(t,\om)$}; 
	\node at (0,1.8) () {\tiny $(t,-\om)$}; 
	\draw[kernel] (farabove) to (above); 
	\draw[kepsilon] (farabove) to (aboveright); 
	\draw[kepsilon] (farabove) to (farleft); 
	\draw[kernel] (farbelow) to (below); 
	\draw[kepsilon] (farbelow) to (belowright); 
	\draw[kepsilon] (farbelow) to (farleft); 
	\draw[kernel] (above) to (left); 
	\draw[kernel] (above) to (aboveright); 
	\draw[kernel] (below) to (left); 
	\draw[kernel] (below) to (belowright);  
	\end{tikzpicture} \phantom{1} .
\end{equation}
This graph is more complicated, and requires more careful treatment. We integrate first the innermost time variable, with two arrows pointing to it labelled $\om_1$ and $-\om_1$ respectively. We bound this contribution by $\scal{\om_1}^{-2}$, uniformly over $u$ and $u'$, as in \eqref{eq:graph_free}.  We are more careful with the integration of the time variable with incoming arrows labelled $\om_2$ and $-\om_2$, and evaluate its contribution to be
\begin{equation*}  
\frac{e^{-|t-u| \scal{\om_2}^2}}{2 \scal{\om_2}^2},
\end{equation*}
as in \eqref{e.graph1.diff.times}. We next compute the contribution of the triangle in the lower part of the diagram by integrating over $u$:
\begin{equation*}  
\int_\R \hP_{t-u}(\om_1 + \om_2) \, \frac{e^{-|t-u| \scal{\om_2}^2}}{2 \scal{\om_2}^2} \, \d u = \frac{1}{{2 \scal{\om_2}^2}} \, \frac 1 {\scal{\om_{2}}^{2} + \scal{\om_{1} + \om_{2}}^{2}}.
\end{equation*}
A similar calculation applies to the upper part of the diagram, and we therefore obtain the bound
\begin{multline}
\label{e.bigsum.222}
	\E \Ll[ |\hat{\<22p>}^{(2)}(t,\om)|^2 \Rr]\\
	 \ls \sum \bigg( \scal{\om_{1}} \scal{\om_{2}} \scal{\om_{2}'} \scal{\om_{3}} \big(\scal{\om_{2}} + \scal{\om_{1} + \om_{2}} \big) \big(\scal{\om_{2}'} + \scal{\om_{2}' - \om_{1}} \big) \bigg)^{-2},
\end{multline}
where the sum is over all $(\om_1,\om_2,\om_2',\om_3)$ satisfying
\begin{equation}  
\label{eq:range_222}
\om_{1} + \om_{3} = \om, \quad \om_{1} + \om_{2} \sim \om_{3} - \om_{2}, \quad -\om_{1} + \om_{2}' \sim - \om_{3} - \om_{2}'.
\end{equation}
The first requirement comes from the ``Kirchhoff" law in the bottom node, while the other two constraints come from the paraproduct in the bottom and upper-most nodes.

We proceed to estimate this sum. 
Note first that the first two conditions in \eqref{eq:range_222} imply that
\begin{equation*}  
\om_1 + \om_2 \sim \om - \Ll( \om_1 + \om_2 \Rr) .
\end{equation*}
By \eqref{qu4}, if $|\om_1 + \om_2| > \frac 8 3$, we deduce that
\begin{equation*}  
|\om_1 + \om_2| \ge c \Ll( |\om| - |\om_1 + \om_2| \Rr),
\end{equation*}
where $c = \frac 9 {64}$, and therefore that
\begin{equation*}  
|\om_1 + \om_2| \ge \frac c{1+c} |\om|.
\end{equation*}
After reducing the constant $c > 0$ as necessary, if follows that in every case, the first two conditions in \eqref{eq:range_222} imply that
\begin{equation}  
\label{e.om1om2om}
\scal{\om_1 + \om_2} \ge 2 c \scal{\om}.
\end{equation}
(The factor of $2$ is of course a matter of convenience only.)
The same argument also shows that under the conditions in \eqref{eq:range_222}, we have
\begin{equation}  
\label{e.om1om2'om}
\scal{\om_1 - \om_2'} \ge 2 c \scal{\om}. 
\end{equation}
Define the sets of indices $E_{1}(\omega)$ and $E_{2}(\omega)$ by
\begin{align*}
	E_{1}(\om) & := \Ll\{(\om_1,\om_2,\om_2',\om_3) \ : \ \text{\eqref{eq:range_222} holds and } \Ll( |\om_{1}| < c {|\om|} \  \text{ or } \ |\om_{3}| < c{|\om|} \Rr) \Rr\}, 
	\\
	E_{2}(\om) & := \Ll\{(\om_1,\om_2,\om_2',\om_3) \ : \ \text{\eqref{eq:range_222} holds and } |\om_{1}| \geq c {|\om|}\text{ and }|\om_{3}| \geq c {|\om|} \Rr\}. 
\end{align*}
We first estimate the sum on the right side of \eqref{e.bigsum.222} over the set $E_1(\om)$. By \eqref{e.om1om2om}, \eqref{e.om1om2'om} and the constraint $\om_3 - \om_2 \sim \om_1 + \om_2$, for variables in the set $E_{1}(\om)$, we must have
\begin{align*}
	\scal{\om_{2}} \geq \td c \, \scal{\om} \quad \text{ and } \quad \scal{\om_{2}'} \ge \td c \, \scal{\om},
\end{align*}
for some $\td c > 0$. 
Thus, using the simple bounds
\begin{align*}
\scal{\om_{2}} + \scal{\om_{1} + \om_{2}} \ge \scal{\om_{2}}, \qquad  \scal{\om_{2}'} + \scal{\om_{2}' - \om_{1}} \ge \scal{\om_{2}'} ,
\end{align*}
and summing over $\om_{1}$ and $\om_{3} = \om - \om_1$ using Lemma~\ref{l41}, we arrive at
\begin{equation} \label{eq:222_e1}
\sum_{E_1(\om)} \cdots \ \ls \scal{\om}^{-1} \sum_{|\om_{2}|, |\om_{2}'| \gtrsim |\om|} \bigg( \scal{\om_{2}}^{-4} \cdot \scal{\om_{2}'}^{-4} \bigg) \ls \scal{\om}^{-3}, 
\end{equation}
where the left side above stands for the sum on the right side of \eqref{e.bigsum.222} restricted to the index set $E_1(\om)$. As for the the index set $E_{2}(\om)$, we start by summing over $\om_{2}$ to get
\begin{align*}
	\sum_{\om_{2}} \scal{\om_{2}}^{-2} \big(\scal{\om_{2}} + \scal{\om_{1} + \om_{2}}\big)^{-2} \ls \scal{\om_{1}}^{-2} \sum_{|\om_{2}| \leq \frac{|\om_{1}|}{2}} \scal{\om_{2}}^{-2} + \sum_{|\om_{2}| \geq \frac{|\om_{1}|}{2}} \scal{\om_{2}}^{-4} \ls \scal{\om_{1}}^{-1}, 
\end{align*}
and similarly we have
\begin{align*}
	\sum_{\om_{2}'} \scal{\om_{2}'}^{-2} \big(\scal{\om_{2}'} + \scal{\om_{2}' - \om_{1}} \big)^{-2} \ls \scal{\om_{1}}^{-1}. 
\end{align*}
We deduce that
\begin{align*}
\sum_{E_2(\om)} \cdots \  \ls \sum_{\substack{\om_{1} + \om_{3} = \om \\ |\om_{1}|, |\om_{3}| \geq \frac{|\om|}{16}}} \scal{\om_{1}}^{-4} \scal{\om_{3}}^{-2}  \ls \scal{\om}^{-2} \sum_{|\om_{1}| \geq \frac{|\om|}{16}} \scal{\om_{1}}^{-4} \ls \scal{\om}^{-3}, 
\end{align*}
with the same notational convention as in \eqref{eq:222_e1}. Combining this with \eqref{e.bigsum.222} and \eqref{eq:222_e1} gives the desired bound for $\hat{\<22p>}^{(2)}$.

\smallskip 

\textbf{Case $\tau = \<32p>$.}
This is the last diagram. We have the Wiener chaos decomposition
\begin{align*}
\hat{\<32p>}(t,\om) &= \phantom{1}
\begin{tikzpicture}[scale=0.7,baseline=0cm]
\node at (0,0) [dot] (middle) {}; 
\node at (0,1.2) [circ] (above) {}; 
\node at (-0.7,0.8) [circ] (aboveleft) {}; 
\node at (0.7,0.8) [circ] (aboveright) {}; 
\node at (0,-1) [var] (below) {\tiny $=$}; 
\node at (-0.7,-0.2) [circ] (left) {}; 
\node at (0.7,-0.2) [circ] (right) {}; 
%
\draw[kepsilon] (below) to (left); 
\draw[kepsilon] (below) to (right); 
\draw[kernel] (below) to (middle); 
\draw[kernel] (middle) to (aboveleft); 
\draw[kernel] (middle) to (aboveright); 
\draw[kernel] (middle) to (above); 
\end{tikzpicture}
\phantom{1} + \phantom{1} 6 \times
\begin{tikzpicture}[scale=0.7,baseline=-0.2cm]
\node at (0,0) [dot] (middle) {}; 
\node at (-0.7,0.7) [circ] (aboveleft) {}; 
\node at (0.7,0.7) [circ] (aboveright) {}; 
\node at (-0.7,-0.5) [circ] (left) {}; 
\node at (0.7,-0.5) [dot] (right) {}; 
\node at (0,-1.2) [var] (below) {\tiny $=$}; 
\draw[kernel] (middle) to (aboveleft); 
\draw[kernel] (middle) to (aboveright); 
\draw[kernel] (middle) to (right); 
\draw[kepsilon] (below) to (left); 
\draw[kepsilon] (below) to (right); 
\draw[kernel] (below) to (middle); 
\end{tikzpicture}
\phantom{1} + \phantom{1} 6 \times \bigg(
\begin{tikzpicture}[scale=0.6,baseline=0.2cm]
\node at (0,1.8) [circ] (above) {}; 
\node at (0,0.8) [dot] (middle) {}; 
\node at (0,-0.8) [var] (below) {\tiny $=$}; 
\node at (-0.8,0) [dot] (left) {}; 
\node at (0.8,0) [dot] (right) {}; 
\node at (0,-1.2) () {\tiny $(t,\om)$}; 
\draw[kernel] (middle) to (above); 
\draw[kernel] (below) to (middle); 
\draw[kepsilon] (below) to (left); 
\draw[kepsilon] (below) to (right); 
\draw[kernel] (middle) to (left); 
\draw[kernel] (middle) to (right); 
\end{tikzpicture}
\phantom{1} - \phantom{1}
\begin{tikzpicture}[scale=0.6,baseline=-0.2cm]
\node at (0,0.8) [dot] (middle) {}; 
\node at (0,-0.8) [var] (below) {\tiny $=$}; 
\node at (-0.8,0) [dot] (left) {}; 
\node at (0.8,0) [dot] (right) {}; 
\node at (0,-1.2) () {\tiny $(t,0)$}; 
\node at (0,1.1) () {}; 
\draw[kernel] (below) to (middle); 
\draw[kepsilon] (below) to (left); 
\draw[kepsilon] (below) to (right); 
\draw[kernel] (middle) to (left); 
\draw[kernel] (middle) to (right); 
\end{tikzpicture}
\cdot
\begin{tikzpicture}[scale=0.7,baseline=-0.2cm]
\node at (0,1) [circ] (above) {}; 
\node at (0,-1) [root] (below) {}; 
\node at (0,-1.4) () {\tiny $(t,\om)$}; 
\draw[kernel] (below) to (above) {}; 
\end{tikzpicture}
\bigg) \\
&=: \hat{\<32p>}^{(5)}(t,\om) + \hat{\<32p>}^{(3)}(t,\om) + \hat{\<32p>}^{(1)}(t,\om), 
\end{align*}
where the last term subtracted in the parenthesis above corresponds to the renormalisation $\cc_n'$ in \eqref{eq:processes} (in the limit as $n \rightarrow +\infty$). As we will see, neither term in the parenthesis above makes sense separately: they both represent some divergent object, but their difference converges to a well-defined limit as the Fourier mode cut-off $n$ goes to infinity. In addition, this limit can be characterised explicitly without referring to a limiting procedure, so this justifies the notation $\hat{\<32p>}^{(1)}(t, \om)$. 

We start with $ \hat{\<32p>}^{(5)}$. Proceeding as in \eqref{eq:22_4}-\eqref{qu9}, using Lemma~\ref{l41} and Lemma~\ref{le:circ_convolution}, we immediately have
\begin{align*}
\E\Ll[ |\hat{\<32p>}^{(5)}(t, \om)|^{2}\Rr] \ls \phantom{1}
\begin{tikzpicture}[scale=0.7,baseline=-0.2cm]
\node at (0,0.8) [dot] (above) {}; 
\node at (0,-0.8) [dot] (below) {}; 
\node at (0,0) [dot] (middle) {}; 
\node at (-0.8,0) [dot] (left) {}; 
\node at (0.8,0) [dot] (right) {}; 
\node at (0,1.8) [var] (farabove) {\tiny $=$}; 
\node at (0,-1.8) [var] (farbelow) {\tiny $=$}; 
\node at (1.8,0) [dot] (farright) {}; 
\node at (-1.8,0) [dot] (farleft) {}; 
\node at (0,2.2) [] {\tiny $(t,-\om)$}; 
\node at (0,-2.2) [] {\tiny $(t,\om)$}; 
\node at (0.3,-0.8) () {\tiny $u$}; 
\node at (0.4,0.8) () {\tiny $u'$}; 
\draw[kernel] (above) to (left); 
\draw[kernel] (above) to (right); 
\draw[kernel] (above) to (middle); 
\draw[kernel] (below) to (middle); 
\draw[kernel] (below) to (left); 
\draw[kernel] (below) to (right); 
\draw[kernel] (farabove) to (above); 
\draw[kepsilon] (farabove) to (farright); 
\draw[kepsilon] (farabove) to (farleft); 
\draw[kernel] (farbelow) to (below); 
\draw[kepsilon] (farbelow) to (farright); 
\draw[kepsilon] (farbelow) to (farleft); 
\end{tikzpicture}
\phantom{1} \ls 
\sum_{\substack{\om_1 + \om_2 + \om_3 = \om\\ \om_1 + \om_3 \sim \om_2}} \frac{1}{\scal{\om_1}^2} \frac{1}{\scal{\om_2}^4}\frac{1}{\scal{\om_3}^2} 
\ls \frac{1}{\scal{\om}^2}, 
\end{align*}
where the term $\frac{1}{\scal{\om_2}^{4}}$ comes from the previous bound for the tree $\<30>$. 

We now turn to the component in the third Wiener chaos, whose second moment is bounded by the graph
\begin{align*}
\E \Ll[ |\hat{\<32p>}^{(3)}(t, \om)|^{2} \Rr]\phantom{1} \ls \phantom{1}
\begin{tikzpicture}[scale=1.2,baseline=-0.2cm]
\node at (0,0.5) [dot] (above) {}; 
\node at (0,-0.5) [dot] (below) {}; 
\node at (0,1.5) [var] (farabove) {\tiny $=$}; 
\node at (0,-1.5) [var] (farbelow) {\tiny $=$}; 
\node at (0.5,0) [dot] (right) {}; 
\node at (-0.5,0) [dot] (left) {}; 
\node at (-1.5,0) [dot] (farleft) {}; 
\node at (1,0.5) [dot] (aboveright) {}; 
\node at (1,-0.5) [dot] (belowright) {}; 
\node at (-0.2, 0.5) () {\tiny $u'$}; 
\node at (-0.2,-0.5) () {\tiny $u$}; 
\node at (-0.4, -0.3) () {\tiny $\om_{1}$}; 
\node at (-0.4, 0.3) () {\tiny $-\om_{1}$}; 
\node at (0.6, -0.25) () {\tiny $\om_{2}$}; 
\node at (0.6, 0.25) () {\tiny $-\om_{2}$}; 
\node at (0.5, -0.65) () {\tiny $\om_{3}$}; 
\node at (0.5, 0.65) () {\tiny $\om_{3}'$}; 
\node at (-0.9, -0.9) () {\tiny $\om_{4}$}; 
\node at (-0.9, 0.9) () {\tiny $-\om_{4}$}; 
\node at (0.7, -1.2) () {\tiny $-\om_{3}$}; 
\node at (0.7, 1.2) () {\tiny $-\om_{3}'$}; 
\draw[kernel] (above) to (right); 
\draw[kernel] (below) to (right); 
\draw[kernel] (farabove) to (above); 
\draw[kepsilon] (farabove) to (aboveright); 
\draw[kepsilon] (farabove) to (farleft); 
\draw[kernel] (farbelow) to (below); 
\draw[kepsilon] (farbelow) to (belowright); 
\draw[kepsilon] (farbelow) to (farleft); 
\draw[kernel] (above) to (left); 
\draw[kernel] (above) to (aboveright); 
\draw[kernel] (below) to (left); 
\draw[kernel] (below) to (belowright);  
\end{tikzpicture}
\phantom{1}. 
\end{align*}
The bound for this graph is similar to the one in \eqref{eq:22_2}, and one can proceed essentially in the same way to get
\begin{align*}
\E \Ll[ |\hat{\<32p>}^{(3)}(t, \om)|^{2} \Rr] \ls \scal{\om}^{-2}, 
\end{align*}
which is the desired bound. 

We now turn to the last term
\begin{align*}
	\hat{\<32p>}^{(1)}(t, \om) = 6 \times \bigg( \phantom{1}
	\begin{tikzpicture}[scale=0.7,baseline=0.2cm]
	\node at (0,1.8) [circ] (above) {}; 
	\node at (0,0.8) [dot] (middle) {}; 
	\node at (0,-0.8) [var] (below) {\tiny $=$}; 
	\node at (-0.8,0) [dot] (left) {}; 
	\node at (0.8,0) [dot] (right) {}; 
	\node at (0,-1.2) () {\tiny $(t,\om)$}; 
	\draw[kernel] (middle) to (above); 
	\draw[kernel] (below) to (middle); 
	\draw[kepsilon] (below) to (left); 
	\draw[kepsilon] (below) to (right); 
	\draw[kernel] (middle) to (left); 
	\draw[kernel] (middle) to (right); 
	\end{tikzpicture}
	\phantom{1} - \phantom{1}
	\begin{tikzpicture}[scale=0.7,baseline=-0.2cm]
	\node at (0,0.8) [dot] (middle) {}; 
	\node at (0,-0.8) [var] (below) {\tiny $=$}; 
	\node at (-0.8,0) [dot] (left) {}; 
	\node at (0.8,0) [dot] (right) {}; 
	\node at (0,-1.2) () {\tiny $(t,0)$}; 
	%
	\draw[kernel] (below) to (middle); 
	\draw[kepsilon] (below) to (left); 
	\draw[kepsilon] (below) to (right); 
	\draw[kernel] (middle) to (left); 
	\draw[kernel] (middle) to (right); 
	\end{tikzpicture}
	\
	\cdot
	\
	\begin{tikzpicture}[scale=0.7,baseline=-0.2cm]
	\node at (0,1) [circ] (above) {}; 
	\node at (0,-1) [root] (below) {}; 
	\node at (0,-1.4) () {\tiny $(t,\om)$}; 
	\draw[kernel] (below) to (above) {}; 
	\end{tikzpicture} \bigg). 
\end{align*}
As mentioned before, the two terms in the parenthesis are both ill-defined, but their difference is well-defined as the limit when the regularisation parameter $n$ tends to infinity. To see this, we introduce a notation for the ``lower square'' which both of the expressions have in common, namely
\begin{align*}
	K_{t-u}(\om)& := \phantom{1}\begin{tikzpicture}[scale=0.7,baseline=-0.2cm]
	\node at (0,0.8) [dot] (middle) {}; 
	\node at (0,-0.8) [var] (below) {\tiny $=$}; 
	\node at (-0.8,0) [dot] (left) {}; 
	\node at (0.8,0) [dot] (right) {}; 
	\node at (0,-1.2) () {\tiny $(t,\om)$}; 
	\node at (0,1.1) () {\tiny $u$}; 
	\draw[kernel] (below) to (middle); 
	\draw[kepsilon] (below) to (left); 
	\draw[kepsilon] (below) to (right); 
	\draw[kernel] (middle) to (left); 
	\draw[kernel] (middle) to (right); 
	\end{tikzpicture} 
	\phantom{1}
	= \sum_{\om_{1} \sim \om_{2}} \hP_{t-u}(\om + \om_{1} + \om_{2}) \Big( \int _{-\infty}^u \hP_{t-u_1}(-\om_1)\hP_{u-u_1}(\om_1) \d u_1 \Big)\\
	& \qquad \times \Big( \int _{-\infty}^u \hP_{t-u_2}(-\om_2)\hP_{u-u_1}(\om_2) \d u_2 \Big) \\
	&=  \sum_{\om_{1}\sim  \om_{2}} \frac{e^{- (t-u) ( \scal{\om_{1}}^{2} + \scal{\om_{2}}^{2} + \scal{\om + \om_{1} + \om_{2}}^{2} )}}{4 \scal{\om_{1}}^{2} \scal{\om_{2}}^{2} }.
\end{align*}
Note that the divergent constant $\cc'$ coincides with $\int_{-\infty}^t K_{t-u}(0) \, \d u$.\footnote{This is slightly formal - here $\cc'$  should denote the limit of $\cc_n'$ as $n \to \infty$, which is infinite as discussed before. We are implicitly assuming that a regularisation is present, although we do not capture it in the notation.}
The kernel $K_{t-u}$ is clearly well defined and controlled uniformly over the regularisation for any fixed $t-u > 0$. 
We use this notation to represent
$\hat{\<32p>}^{(1)}(t,\om)$ as 
\begin{equation} \label{eq:linear_decom}
	\begin{split}
		\hat{\<32p>}^{(1)}(t, \om)  &= \int_{-\infty}^{t} K_{t-u}(\om) \big( \hat{\<1>}(u,\om) - \hat{\<1>}(t,\om) \big) \,  \d u,
	\end{split}
\end{equation}
so that
\begin{align*}
\E \Ll[ |\hat{\<32p>}^{(1)}(t, \om)|^{2} \Rr] & = \int_{-\infty}^{t} \int_{-\infty}^{t} K_{t-u}(\om) K_{t-u'}(\om) \\
 & \qquad \qquad \times\E \Ll[\big( \hat{\<1>}(u,\om) - \hat{\<1>}(t,\om) \big)  \big( \hat{\<1>}(u',-\om) - \hat{\<1>}(t,-\om) \big) \Rr] \, \d u \, \d u'. 
\end{align*}
Now, since  by \eqref{e.graph1.diff.times} we have for any $\lambda>0$
\begin{align*}
\E \Ll[|\hat{\<1>}(u,\om) - \hat{\<1>}(t,\om)|^{2}\Rr]^\frac 1 2  \ls (t-u)^{\lambda} \scal{\om}^{-1+ 2 \lambda}, 
\end{align*}
an application of the Cauchy-Schwarz inequality yields
\begin{align*}
\E \Ll[|\hat{\<32p>}^{(1)}(t, \om)|^{2} \Rr]\ls \scal{\om}^{-2 + 4 \lambda} \bigg( \int_{-\infty}^{t} (t-u)^{\lambda} K_{t-u}(\om) \d u \bigg)^{2} \ls \scal{\om}^{-2 },
\end{align*}
where  we have used the fact that for any (small) strictly positive $\lambda$, the integral can be bounded uniformly over the regularisation by $\scal{\om}^{-2\lambda}$. This latter point can be checked as follows: for $\lambda=0$ and without the condition $\om_1 \sim \om_2$, we had already calculated the integral in \eqref{div.log}; the factor $(t-u)^{\lambda}$ makes an extra power  $(\scal{\om_{1}}^{2} + \scal{\om_{2}}^{2} + \scal{\om + \om_{1} + \om_{2}}^{2} )^{-\lambda}$  appear in the sum, which permits to invoke Lemma~\ref{l41} and conclude.  This completes the bound for $\tau = \<32p>$.

\section{Bounds for time differences} \label{sec:time}

We finally  discuss briefly how the reasoning in Section~\ref{sec:bounds} should be modified to establish the bound \eqref{eq:main_bound2}
on the time differences $\E |\htau(t,\om) - \htau(s,\om)|^{2}$. We illustrate the (simple) modification necessary for the graph $\tau = \<31p>$. First, recall 
the Wiener chaos decomposition \eqref{eq:expression_31} for this graph, which yields the following expression for its time-differences:
\begin{equation*} 
\hat{\<31p>}(t,\om) - \hat{\<31p>}(s,\om) = \phantom{1}
\left( \begin{tikzpicture}[scale=0.5,baseline=-0.3cm]
\node at (0,0) [dot] (middle) {}; 
\node at (0,1) [circ] (above) {};
\node at (-0.7,0.6) [circ] (left) {}; 
\node at (0.7,0.6) [circ] (right) {}; 
\node at (0,-1) [var] (below) {\tiny $=$}; 
\node at (0.8,-0.3) [circ] (middleright) {}; 
\node at (0,-1.4) [] {\tiny $(t,\om)$}; 
\draw[kernel] (middle) to (left); 
\draw[kernel] (middle) to (above); 
\draw[kernel] (middle) to (right); 
\draw[kernel] (below) to (middle); 
\draw[kepsilon] (below) to (middleright); 
\end{tikzpicture}
\phantom{1}
- 
\begin{tikzpicture}[scale=0.5,baseline=-0.3cm]
\node at (0,0) [dot] (middle) {}; 
\node at (0,1) [circ] (above) {};
\node at (-0.7,0.6) [circ] (left) {}; 
\node at (0.7,0.6) [circ] (right) {}; 
\node at (0,-1) [var] (below) {\tiny $=$}; 
\node at (0.8,-0.3) [circ] (middleright) {}; 
\node at (0,-1.4) [] {\tiny $(s,\om)$}; 
\draw[kernel] (middle) to (left); 
\draw[kernel] (middle) to (above); 
\draw[kernel] (middle) to (right); 
\draw[kernel] (below) to (middle); 
\draw[kepsilon] (below) to (middleright); 
\end{tikzpicture}
\phantom{1}
\right)
 + \phantom{1} 3 \times
\left( 
\begin{tikzpicture}[scale=0.5,baseline=-0.3cm]
\node at (0,0) [dot] (middle) {}; 
\node at (-0.7,0.8) [circ] (aboveleft) {}; 
\node at (0.7,0.8) [circ] (aboveright) {}; 
\node at (0,-1.2) [var] (below) {\tiny $=$}; 
\node at (1,-0.6) [dot] (right) {}; 
\node at (0,-1.6) [] {\tiny $(t,\om)$}; 
\draw[kernel] (below) to (middle); 
\draw[kepsilon] (below) to (right); 
\draw[kernel] (middle) to (right); 
\draw[kernel] (middle) to (aboveleft); 
\draw[kernel] (middle) to (aboveright); 
\end{tikzpicture}
-
\begin{tikzpicture}[scale=0.5,baseline=-0.3cm]
\node at (0,0) [dot] (middle) {}; 
\node at (-0.7,0.8) [circ] (aboveleft) {}; 
\node at (0.7,0.8) [circ] (aboveright) {}; 
\node at (0,-1.2) [var] (below) {\tiny $=$}; 
\node at (1,-0.6) [dot] (right) {}; 
\node at (0,-1.6) [] {\tiny $(s,\om)$}; 
\draw[kernel] (below) to (middle); 
\draw[kepsilon] (below) to (right); 
\draw[kernel] (middle) to (right); 
\draw[kernel] (middle) to (aboveleft); 
\draw[kernel] (middle) to (aboveright); 
\end{tikzpicture}
\right).
\end{equation*}
The differences of graphs can be bounded separately. We only discuss the first difference here. We can rewrite this difference as
\begin{equation*} 
\begin{tikzpicture}[scale=0.7,baseline=-0.3cm]
\node at (0,0) [dot] (middle) {}; 
\node at (0,1) [circ] (above) {};
\node at (-0.7,0.6) [circ] (left) {}; 
\node at (0.7,0.6) [circ] (right) {}; 
\node at (-1.2,-0.6) [] {\tiny{$P_{t-u}-P_{s-u}$}};
\node at (1.3,-0.6) [] {\tiny{$P_{t-u_1}$}};
\node at (0,-1) [var] (below) {\tiny $=$}; 
\node at (0.8,-0.3) [circ] (middleright) {}; 
%
\draw[kernel] (middle) to (left); 
\draw[kernel] (middle) to (above); 
\draw[kernel] (middle) to (right); 
\draw[kernel] (below) to (middle); 
\draw[kepsilon] (below) to (middleright); 
\end{tikzpicture}
\phantom{1}
+
\begin{tikzpicture}[scale=0.7,baseline=-0.3cm]
\node at (0,0) [dot] (middle) {}; 
\node at (0,1) [circ] (above) {};
\node at (-0.7,0.6) [circ] (left) {}; 
\node at (0.7,0.6) [circ] (right) {}; 
\node at (0,-1) [var] (below) {\tiny $=$}; 
\node at (0.8,-0.3) [circ] (middleright) {}; 
\node at (-.5,-0.6) [] {\tiny{$P_{s-u}$}};
\node at (2,-0.6) [] {\tiny{$P_{t-u_1}- P_{s-u_1}$}};
%
\draw[kernel] (middle) to (left); 
\draw[kernel] (middle) to (above); 
\draw[kernel] (middle) to (right); 
\draw[kernel] (below) to (middle); 
\draw[kepsilon] (below) to (middleright); 
\end{tikzpicture}
\phantom{1},
\end{equation*}
where we have made explicit which kernels are associated to the lower edges in the graph. (A more systematic treatment would suggest the use of new graphical notation for these!) The variance of each of these terms can then be 
bounded, by ``glueing'' two copies of each graph together, as in \eqref{eq:31_4}. We then use the elementary bounds
\begin{align*}
\int_{-\infty}^t |\hP_{t-u}(\om) - \hP_{s-u}(\om)| \, \d u  &\ls \frac{1}{\scal{\om}^2} \big(1 \wedge |t-s| \scal{\om}^2 \big),\\ 
\int_{-\infty}^t  (\hP_{t-u_1}(\om) - \hP_{s-u_1}(\om) )^2 \, \d u_1  &\ls \frac{1}{\scal{\om}^{2}} \big( 1 \wedge |t-s| \scal{\om}^{2} \big).
\end{align*}
By interpolation, the right side can be replaced by $\ls \scal{\om}^{-2+2\lambda} |t-s|^\lambda$, for any $\lambda \in [0,1]$. In other words, an extra factor $|t-s|^{\lambda}$ can be 
obtained by sacrificing a bit of the decay of the integral in $\om$. Then all of the arguments based on convolutions can be performed 
exactly as before, only with a slightly worse factor of $\om$ at one place. We do not go through the details here, but leave it to the interested
reader to check that this does not change the arguments in Section~\ref{sec:bounds}  in any significant way.

\appendix

\section{Alternative proof of Nelson's estimate}

We now give a second proof of Proposition \ref{p.nelson}, based on the following logarithmic Sobolev inequality (see  \cite[Section 1.6]{Bogachev} and \cite[Sections~1.1 and 1.5]{Nualart}).

\begin{lemma} [log-Sobolev inequality]
	\label{le:log-sobolev}
	Let $\mu$ be a Gaussian measure and $X \in W^{1,2}(\mu)$. We have
	\begin{equation} \label{eq:log-sobolev}
	\E (|X|^{2} \log |X|) \leq \E |\D X|^{2} + \frac{1}{2} \E |X|^{2} \log (\E |X|^{2}), 
	\end{equation}
	where $\D$ is the Malliavin derivative, and $\E$ is the expectation taken with respect to $\mu$. 
\end{lemma}

Now, let $T_{t}$ be the Ornstein-Uhlenbeck semigroup defined by
\begin{equation}
T_{t} X = \sum_{n=0}^{+\infty} e^{-nt} X_{n}, 
\end{equation}
where $X_{n}$ is the component of $X$ in $\mcl H_{n}$. The Ornstein-Uhlenbeck semigroup 
is closely related to the Malliavin derivative, because it determines the quadratic form associated with the infinitesimal generator $L$ of $T_t$. More precisely, 
for sufficiently nice random variables $X, Y$, we have
\begin{equation}\label{integrationbypart}
\partial_t \E[(T_t X) Y] = \E[(LX) Y] = - \E[ \langle \D X, \D Y \rangle ].
\end{equation}
See \cite[Section~1.4]{Nualart} for a more detailed discussion of these objects.
The main use of the logarithmic Sobolev inequality will be to show the following hypercontractivity estimate. 

\begin{proposition} [Hypercontractivity]
	\label{p.hypercontractivity}
	Let $T_{t}$ be the Ornstein-Uhlenbeck semigroup. We have
	\begin{equation} \label{eq:hypercontractivity}
	\big( \E |T_{t} X|^{q} \big)^{\frac{1}{q}} \leq \big( \E |X|^{p} \big)^{\frac{1}{p}},
	\end{equation}
	for all $p \geq 2$ and $q = 1 + (p-1) e^{2t}$. 
\end{proposition}

\begin{proof} [Second proof of Proposition \ref{p.nelson}]
	If $X \in \mcl H_{n}$, then $T_{t} X = e^{-nt} X$, and we can see that Proposition \ref{p.nelson} is an immediate consequence of \eqref{p.hypercontractivity}. It then remains to prove Proposition \ref{p.hypercontractivity}. We can assume $X \geq 0$ without loss of generality. 
	
	Fix $p \geq 2$. Let $q(t) = 1 + (p-1) e^{2t}$, and let
	\begin{align*}
	F(t) = \E |T_{t} X|^{q(t)}, \qquad G(t) = F(t)^{\frac{1}{q(t)}}. 
	\end{align*}
	Our aim is to show that $G'(t) \leq 0$ for all $t > 0$, and \eqref{eq:hypercontractivity} will follow. In fact, we have
	\begin{align*}
	G'(t) = G(t) \bigg[ - \frac{q'(t)}{q^{2}(t)} \log F(t) + \frac{F'(t)}{q(t) F(t)} \bigg]. 
	\end{align*}
	Since $q'(t) \geq 0$, it suffices to show that
	\begin{equation} \label{eq:intermediate}
	-\frac{1}{q(t)} F(t) \log F(t) + \frac{F'(t)}{q'(t)} \leq 0. 
	\end{equation}
	Noting that
	\begin{align*}
	F'(t) = \E \bigg[ (T_{t}X)^{q(t)} \bigg( q'(t) \log(T_{t}X) + q(t) \frac{LT_{t}X}{T_{t}X} \bigg)  \bigg], 
	\end{align*}
	we see that \eqref{eq:intermediate} is equivalent to
	\begin{equation} \label{eq:hyper_before_log}
	-\frac{1}{q(t)} F(t) \log F(t) + \E \bigg[ (T_{t}X)^{q(t)} \log (T_{t}X) \bigg] + \frac{q(t)}{q'(t)} \E \bigg[ (T_{t}X)^{q(t)-1} (LT_{t}X) \bigg]. 
	\end{equation}
	 Applying the log-Sobolev inequality to the random variable $(T_{t}X)^{\frac{q(t)}{2}}$ and using the integration by parts formula \eqref{integrationbypart} which in the current context becomes
	\begin{align*}
	\E \langle \D Y, \D Z \rangle = - \E \big(Y (LZ) \big), 
	\end{align*}
	we see that \eqref{eq:hyper_before_log} follows immediately. 
\end{proof}

\subsection*{Acknowledgments} We are grateful to the very careful referees for their detailed and constructive criticisms, which led to many improvements over the whole paper. We were particularly impressed that one of the referees' reports turned out to be longer than our paper itself! JCM is partially supported by the ANR Grant LSD (ANR-15-CE40-0020-03).
HW is supported by the Royal Society through the University Research Fellowship UF140187. WX is supported by the EPSRC through the fellowship EP/N021568/1.

\bibliographystyle{abbrv}
\bibliography{diagrams}

\end{document}